\documentclass{amsart}
\usepackage[dvips]{epsfig}
\usepackage{graphics}
\usepackage{latexsym}
\usepackage{amsmath}
\usepackage{amsthm}
\usepackage{amssymb}

\newtheorem{thm}{Theorem}
\newtheorem{lem}[thm]{Lemma}

\newtheorem{defi}[thm]{Definition}

\newtheorem{prop}[thm]{Proposition}
\newtheorem{rk}[thm]{Remark}
\newtheorem{nota}[thm]{Notation}

\newenvironment{preuve}{\vip \noindent {\it Proof}}{\hfill$\square$\vip}

\newcommand{\poubelle}[1]{}
\newcommand{\vip}{\vskip.2cm}
\newcommand{\e}{{\varepsilon}}
\newcommand{\rr}{{\mathbb{R}}}
\newcommand{\rd}{{\rr^d}}
\newcommand{\nn}{{\mathbb{N}}}
\newcommand{\E}{\mathbb{E}}
\renewcommand{\P}{\mathbb{P}}
\newcommand{\Var}{\mathbb{V}{\rm ar}\,}
\newcommand{\Cov}{\mathbb{C}{\rm ov}\,}
\newcommand{\cH}{{\mathcal H}}
\newcommand{\cA}{{\mathcal A}}
\newcommand{\cW}{{\mathcal W}}
\newcommand{\cD}{{\mathcal D}}
\newcommand{\cB}{{\mathcal B}}

\newcommand{\cT}{{\mathcal T}}
\newcommand{\cP}{{\mathcal P}}
\newcommand{\cR}{{\mathcal R}}
\newcommand{\cE}{{\mathcal E}}
\newcommand{\indiq}{{{\bf 1}}}
\newcommand{\intrd}{{\int_{\rd}}}

\begin{document}

\title[Convergence of the empirical measure]
{On the rate of convergence in Wasserstein distance of the empirical measure}

\author{Nicolas Fournier}

\address{Nicolas Fournier, Laboratoire de Probabilit\'es et Mod\`eles al\'eatoires, UMR 7599,
UPMC, Case 188, 4 pl. Jussieu, F-75252 Paris Cedex 5, France, 
E-mail: {\tt nicolas.fournier@upmc.fr}}

\author{Arnaud Guillin}
\address{Arnaud Guillin, Laboratoire de Math\'ematiques, UMR 6620, Universit\'e Blaise Pascal,
Av. des landais, 63177 Aubiere cedex, E-mail: {\tt arnaud.guillin@math.univ-bpclermont.fr}}

\begin{abstract}
Let $\mu_N$ be the empirical measure associated to a $N$-sample of a given
probability distribution $\mu$ on $\mathbb{R}^d$. 
We are interested in the rate of convergence of $\mu_N$ to $\mu$, when
measured in the Wasserstein distance of order $p>0$.
We provide some satisfying non-asymptotic $L^p$-bounds
and concentration inequalities,
for any values of $p>0$ and $d\geq 1$. We extend also the non asymptotic $L^p$-bounds to stationary $\rho$-mixing sequences, Markov chains,
and to some interacting particle systems.
\end{abstract}

\maketitle



\section{Introduction and results}

\subsection{Notation}

Let $d\geq 1$ and $\cP(\rd)$ stand for the set of all probability measures on $\rd$.
For $\mu \in \cP(\rd)$, we consider an i.i.d. sequence $(X_k)_{k\geq 1}$
of $\mu$-distributed random variables and, for $N \geq 1$, the empirical measure $$\mu_N:=\frac1N
\sum_{k=1}^N \delta_{X_k}.$$
As is well-known, by Glivenko-Cantelli's theorem, $\mu_N$ tends weakly to $\mu$ as $N\to \infty$ 
(for example in probability, see Van der Vaart-Wellner \cite{vdvw} for details and various modes 
of convergence).  The aim of the paper is to quantify this convergence,
when the error is measured in some Wasserstein distance.
Let us set, for $p\geq 1$ and $\mu,\nu$ in $\cP(\rd)$,
$$
\cT_p(\mu,\nu) = \inf \left\{\left(\int_{\rd\times\rd} |x-y|^p \xi(dx,dy)  \right)
\; : \; \xi \in \cH(\mu,\nu) \right\},
$$
where $\cH(\mu,\nu)$ is the set of all probability measures on $\rd\times\rd$ with marginals
$\mu$ and $\nu$. See Villani \cite{v} for a detailed study of $\cT_p$. The 
Wasserstein distance $\cW_p$ on $\cP(\rd)$ is defined by $\cW_p(\mu,\nu)=\cT_p(\mu,\nu)$
if $p\in (0,1]$ and $\cW_p(\mu,\nu)=(\cT_p(\mu,\nu))^{1/p}$ if $p>1$.

\poubelle{
\vip

It is a major challenge in quantization theory to develop a codebook of size $N$ with the smallest 
possible distance (or averaged distance) to a $\mu$-distributed point, for example in signal theory
(see Dereich \cite{der} and Dereich-Scheutzow-Schottstedt \cite{dss}), in statistics
(see e.g. Biau-Devroye-Lugosi \cite{bdl} and Lalo\"e \cite{l}) or in finance 
(see Pag\`es-Wilbertz \cite{pw}). A very close topic is the matching problem
(quantify the distance between two independent samples) as in the famous paper \cite{akt}
by Ajtai-Koml\'os-Tusn\`ady, see also the recent results of Barthe-Bordenave \cite{bb}
with applications to the salesperson problem. We refer to these nice papers for an 
extensive introduction on this vast topic.
}

\vip

The present paper studies the rate of convergence
to zero of $\cT_p(\mu_N,\mu)$. This can be done in an asymptotic way, finding e.g.
a sequence $\alpha(N)\to 0$ such that 
$\lim_N \alpha(N)^{-1}\cT_p(\mu_N,\mu) <\infty$ a.s. or $\lim_N\alpha(N)^{-1}
\E(\cT_p(\mu_N,\mu))<\infty$. Here we will rather derive some non-asymptotic moment estimates
such as
$$
\E(\cT_p(\mu_N,\mu))\le \alpha(N) \quad \hbox{for all $N\geq 1$}
$$
as well as some non-asymptotic concentration estimates (also often called deviation inequalities)
$$
\Pr(\cT_p(\mu_N,\mu)\geq x)\le\alpha(N,x) \quad \hbox{for all $N\geq 1$, all $x>0$.}
$$
They are naturally related to moment (or exponential moment) conditions on the law $\mu$
and we hope to derive an interesting interplay between the dimension 
$d\geq 1$, the cost parameter $p>0$ and these moment conditions. 
Let us introduce precisely these moment conditions. 
For $q>0$, $\alpha>0$, $\gamma>0$ and $\mu \in \cP(\rd)$, we define
$$
M_q(\mu):= \intrd |x|^q \mu(dx) \quad \hbox{and} \quad \cE_{\alpha,\gamma}(\mu):= \intrd 
e^{\gamma|x|^\alpha} \mu(dx).
$$
We now present our main estimates, the
comparison with the existing results and methods will be developped after 
this presentation. Let us however mention at once that our paper relies on some recent ideas
of Dereich-Scheutzow-Schottstedt \cite{dss}.

\subsection{Moment estimates}

We first give some $L^p$ bounds.

\begin{thm}\label{mainmom}
Let $\mu \in \cP(\rd)$ and let $p>0$. Assume that $M_q(\mu)<\infty$ for some $q>p$.
There exists a constant $C$ depending only on $p,d,q$ such that, for all $N\geq 1$,
$$
\E\left(\cT_p(\mu_N,\mu)\right) \leq
C M_q^{p/q}(\mu)\left\{\begin{array}{ll}
N^{-1/2} +N^{-(q-p)/q}& \!\!\!\hbox{if $p>d/2$ and $q\ne 2p$},  \\[+3pt]
N^{-1/2} \log(1+N)+N^{-(q-p)/q} &\!\!\! \hbox{if $p=d/2$ and $q\ne 2p$}, \\[+3pt]
N^{-p/d}+N^{-(q-p)/q} &\!\!\!\hbox{if $p\in (0,d/2)$ and $q\ne d/(d-p)$}.
\end{array}\right.
$$
\end{thm}

Observe that when $\mu$ has sufficiently many moments (namely if 
$q>2p$ when $p\geq d/2$ and $q> dp/(d-p)$ when $p\in (0,d/2)$),
the term $N^{-(q-p)/q}$ is small and can be removed.
We could easily treat, for example, the case $p>d/2$ and $q=2p$ but this would lead to some
logarithmic terms and the paper is technical enough. 

\vip

This generalizes \cite{dss}, in which only the case
$p\in [1,d/2)$ (whence $d\geq 3$) and $q>dp/(d-p)$ was treated.  The argument is also slightly
simplified.

\vip

To show that Theorem \ref{mainmom} is really sharp, let us give examples where lower bounds can
be derived quite precisely.

\vip

(a) If $a\ne b \in \rd$ and $\mu=(\delta_a+\delta_b)/2$, one easily checks (see e.g. 
\cite[Remark 1]{dss}) that $\E(\cT_p(\mu_N,\mu)) \geq c N^{-1/2}$ for all $p\geq 1$.
Indeed, we have $\mu_N=Z_N\delta_a+(1-Z_N)\delta_b$ with $Z_N=N^{-1}\sum_1^N \indiq_{\{X_i=a\}}$,
so that $\cT_p(\mu_N,\mu)=|a-b|^p|Z_N-1/2|$, of which the expectation is of order $N^{-1/2}$.

\vip

(b) Such a lower bound in $N^{-1/2}$ can easily be extended to any $\mu$ (possibly very smooth)
of which the support is of the form $A\cup B$ with $d(A,B)>0$
(simply note that $\cT_p(\mu_N,\mu)\geq 
d^p(A,B)|Z_N-\mu(A)|$, where $Z_N=N^{-1}\sum_1^N \indiq_{\{X_i \in A\}}$).

\vip

(c) If $\mu$ is the uniform distribution on $[-1,1]^d$, it is well-known and not difficult to prove
that for $p>0$, $\E(\cT_p(\mu_N,\mu)) \geq c N^{-p/d}$. 
Indeed, consider a partition of $[-1,1]^d$ into (roughly) $N$ cubes with length $N^{-1/d}$.
A quick comptation shows that with probability greater than some $c>0$ (uniformly in $N$),
half of these cubes will not be charged by $\mu_N$. But on this event,
we clearly have $\cT_p(\mu_N,\mu)\geq a N^{-1/d}$ for some $a>0$, because each time a cube is
not charged by $\mu_N$, a (fixed) proportion 
of the mass of $\mu$ (in this cube) is at distance at least
$N^{-1/d}/2$ of the support of $\mu_N$. One easily concludes. 

\vip

(d) When $p=d/2=1$, it has been shown by Ajtai-Koml\'os-Tusn\'ady \cite{akt} 
that for $\mu$ the uniform measure on $[-1,1]^d$,
$\cT_1(\mu_N,\mu) \simeq c (\log N / N)^{1/2}$ with high probability, implying that 
$\E(\cT_1(\mu_N,\mu)) \geq c (\log N / N)^{1/2}$.

\vip

(e) Let $\mu(dx)= c |x|^{-q-d}\indiq_{\{|x|\geq 1\}}dx$ for some $q>0$. Then $M_r(\mu)<\infty$
for all $r\in(0,q)$ and for all $p\geq 1$, $\E(\cT_p(\mu_N,\mu)) \geq c N^{-(q-p)/q}$.
Indeed, $\P(\mu_N(\{|x|\geq N^{1/q}\})=0)= (\mu(\{|x|< N^{1/q}\}))^N
=(1-c/N)^N \geq c>0$ and $\mu(\{|x|\geq 2 N^{1/q}\})\geq c/N$.
One easily gets convinced that 
$\cT_p(\mu_N,\mu)\geq N^{p/q}\indiq_{\{\mu_N(\{|x|\geq N^{1/q}\})=0\}} \mu(\{|x|\geq 2 N^{1/q}\})$,
from which the claim follows.

\vip

As far as {\it general laws} are concerned, Theorem \ref{mainmom} is really sharp:
the only possible improvements
are the following. The first one, quite interesting, 
would be to replace $\log (1+N)$ by something like $\sqrt{\log (1+N)}$ when $p=d/2$
(see point (d) above). It is however not clear it is feasible in full generality. The second 
one, which should be a mere (and not very interesting) refinement,
would be to sharpen the bound in $N^{-(q-p)/q}$ when $M_q(\mu)<\infty$: 
point (e) only shows that there is $\mu$ with $M_q(\mu)<\infty$ for which we have a lowerbound in 
$N^{-(q-p)/q - \e}$ for all $\e>0$.

\vip

However, some improvements are possible when restricting the class of laws $\mu$.
First, when $\mu$ is the uniform distribution in $[-1,1]^d$, the results of Talagrand \cite{t,t2}
strongly suggest that when $d\geq 3$, $\E(\cT_p(\mu_N,\mu)) \simeq N^{-p/d}$ for all $p>0$,
and this is much better than $N^{-1/2}$ when $p$ is large. 
Such a result would of course immediately extend to any distribution $\mu=\lambda \circ F^{-1}$,
for $\lambda$ the uniform distribution in $[-1,1]^d$ and 
$F:[-1,1]^d\mapsto \rd$ Lipschitz continuous.
In any case, a smoothness assumption for $\mu$ cannot be sufficient, see point (b) above.

\vip

Second, for irregular laws, the convergence can be much faster that $N^{-p/d}$ when 
$p<d/2$, see point (a) above where, in an extreme case, we get $N^{-1/2}$ for all values of $p>0$.
It is shown by Dereich-Scheutzow-Schottstedt \cite{dss} (see also Barthe-Bordenave \cite{bb})
that indeed, for a singular law, $\lim_N N^{-p/d}\E(\cT_p(\mu_N,\mu))=0$.

\subsection{Concentration inequalities}

We next state some concentration inequalities.

\begin{thm}\label{maindev}
Let $\mu \in \cP(\rd)$ and let $p>0$. 
Assume one of the three following conditions:
\begin{align}\label{asexfort}
&\exists\; \alpha>p, \; \exists\;\gamma>0,\;  \cE_{\alpha,\gamma}(\mu)<\infty,\\
\label{asexfaible}
\hbox{or} \quad &\exists\; \alpha\in(0,p), \;  \exists\;\gamma>0,\;  \cE_{\alpha,\gamma}(\mu)<\infty,\\
\label{asmom}
\hbox{or} \quad &\exists\; q>2p , \;  M_q(\mu)<\infty.
\end{align}
Then for all $N\geq 1$, all $x\in (0,\infty)$,
$$
\P(\cT_p(\mu^N,\mu)\geq x) \leq a(N,x)\indiq_{\{x\leq 1\}}+b(N,x),
$$
where
$$
a(N,x)=C \left\{\begin{array}{ll}
\exp(-cNx^2) & \hbox{if $p>d/2$}, \\[+3pt]
\exp(-cN(x/\log(2+1/x))^2) & \hbox{if $p=d/2$}, \\[+3pt]
\exp(-cN x^{d/p}) & \hbox{if $p\in[1,d/2)$}\end{array}\right.
$$
and
$$
b(N,x)=C\left\{\begin{array}{ll}
\exp(-cNx^{\alpha/p})\indiq_{\{x> 1\}} & \hbox{under \eqref{asexfort}}, \\[+3pt]
\exp(-c(Nx)^{(\alpha-\e)/p})\indiq_{\{x\leq 1\}}+\exp(-c (Nx)^{\alpha/p})\indiq_{\{x>1\}} 
& \hbox{$\forall \;\e\in(0,\alpha)$ under \eqref{asexfaible}},\\[+3pt]
N (Nx)^{-(q-\e)/p}& \hbox{$\forall\;\e\in(0,q)$ under \eqref{asmom}}.\end{array}\right.
$$
The positive constants $C$ and $c$ depend only on $p,d$ and either
on $\alpha,\gamma,\cE_{\alpha,\gamma}(\mu)$ (under \eqref{asexfort}) or 
on $\alpha,\gamma,\cE_{\alpha,\gamma}(\mu),\e$ (under \eqref{asexfaible}) 
or on $q,M_q(\mu),\e$ (under \eqref{asmom}).
\end{thm}

We could also treat the  {\it critical} case where $\cE_{\alpha,\gamma}(\mu)<\infty$ with $\alpha=p$,
but the result we could obtain is slightly more intricate and not very satisfying for small 
value of $x$ (even if good for large ones).

\begin{rk} When assuming \eqref{asexfaible} with $\alpha \in (0,p)$, we actually also prove that 
$$
b(N,x)\leq C\exp(-c N x^2 (\log(1+N))^{-\delta})+C\exp(-c (Nx)^{\alpha/p}),
$$
with $\delta=2p/\alpha-1$, see Step 5 of the proof of Lemma \ref{znpex} below.
This allows us to extend the inequality  
$b(N,x)\leq C\exp(-c (Nx)^{\alpha/p})$ to all values of $x\geq x_N$, for some (rather small)
$x_N$ depending on $N,\alpha,p$. But for very small values of $x>0$,
this formula is less interesting than that of Theorem \ref{maindev}.
Despite much effort, we have not been able
to get rid of the logarithmic term.
\end{rk}

We believe that these estimates are quite satisfying. 
To get convinced, first observe that the {\it scales} seem to be the good ones.
Recall that $\E(\cT_p(\mu_N,\mu))=\int_0^\infty \P( \cT_p(\mu^N,\mu)\geq x)dx$.

(a)  One easily checks that $\int_0^\infty a(N,x)dx \leq CN^{-p/d}$ if $p<d/2$, 
$CN^{-1/2}\log (1+N)$ if $p=d/2$, and $CN^{-1/2}$ if $p>d/2$, as in Theorem \ref{mainmom}.

(b) When integrating $b(N,x)$ (or rather $b(N,x)\land 1$), we find $N^{-(q-\e-p)/(q-\e)}$
under \eqref{asmom} and something smaller under \eqref{asexfort} or \eqref{asexfaible}. 
Since we can take $q-\e>2p$,
this is less than $N^{-1/2}$ (and thus also less than $N^{-p/d}$ if $p<d/2$
and than $N^{-1/2}\log (1+N)$ if $p=d/2$).

\vip

The {\it rates of decrease} are also satisfying in most cases.
Recall that in deviation estimates, we never get something better than
$\exp(-Ng(x))$ for some function $g$. Hence $a(N,x)$ is probably optimal.
Next, for $\bar Y_N$ the empirical mean of a family of centered i.i.d. random variables,
it is well-known that the {\it good} deviation inequalities are the following.

(a) If $\E[\exp(a|Y_1|^\beta)]<\infty$ with $\beta\geq 1$, then 
$\Pr[|\bar Y_N|\geq x ] \leq Ce^{-cNx^2}\indiq_{\{x\leq 1\}} + Ce^{-cNx^{\beta}}\indiq_{\{x> 1\}}$,
see for example Djellout-Guillin-Wu \cite{dgw}, Gozlan \cite{goz} or Ledoux \cite{led},
using transportation cost inequalities. 

(b)  If $\E[\exp(a|Y_1|^\beta)]<\infty$ with $\beta<1$, then 
$\Pr[|\bar Y_N|\geq x ] \leq Ce^{-cNx^2} + Ce^{-c(Nx)^{\beta}}$,
see Merlev\`ede-Peligrad-Rio \cite[Formula (1.4)]{mpr} which is based on 
results by Borovkov \cite{boro}. 

(c) If $\E[|Y_1|^r]<\infty$ for some $r>2$, then 
$\Pr[|\bar Y_N|\geq x ] \leq Ce^{-cNx^2} + CN (Nx)^{-r}$, see Fuk-Nagaev \cite{fn}, 
using usual truncation arguments. 

\vip

Our result is in perfect adequation with these facts 
(up to some arbitratry small loss due to $\e$ under \eqref{asexfaible} and \eqref{asmom}) 
since $\cT_p(\mu_N,\mu)$
should behave very roughly as the mean of the $|X_i|^p$'s, which e.g. has an exponential
moment with power $\beta:=\alpha/p$ under \eqref{asexfort} and \eqref{asexfaible}.

\subsection{Comments}

The control of the distance between the empirical measure of an i.i.d. sample and its true 
distribution is of course a long standing problem central both in probability, statistics and 
informatics with a wide number of applications: quantization (see Delattre-Graf-Luschgy-Pag\`es
\cite{dghp} and Pag\`es-Wilbertz \cite{pw} for recent results), 
optimal matching (see Ajtai-Koml\'os-Tusn\'ady \cite{akt}, Dobri\'c-Yukich \cite{dy}, Talagrand
\cite{t2}, Barthe-Bordenave \cite{bb}), density estimation, clustering 
(see Biau-Devroye-Lugosi \cite{bdl} and Lalo\"e \cite{l}), 
MCMC methods (see \cite{roro} for bounds on ergodic averages), particle systems and approximations of partial differential 
equations (see Bolley-Guillin-Villani \cite{bgv} and Fournier-Mischler \cite{fm}). 
We refer to these papers for an extensive introduction on this vast topic.

\vip

If many distances can be used to consider the problem, the Wasserstein distance is quite 
natural, in particular in quantization or for particle approximations of P.D.E.'s. 
However the depth of the problem was discovered only recently by Ajtai-Koml\'os-Tusn\'ady \cite{akt}, who considered 
the uniform measure on the square, investigated thoroughly by Talagrand \cite{t2}.
As a review of the litterature is somewhat impossible, et us just say that the methods 
involved were focused on two methods inherited by the definitions of the Wasserstein distance: the construction of a coupling or by duality to control a particular 
empirical process.

\vip

Concerning moment estimates (as in Theorem \ref{mainmom}), 
some results can be found in
Horowitz-Karandikar \cite{hk}, Rachev-R\"uschendorf \cite{rr} and Mischler-Mouhot \cite{mm}.
But theses results are far from optimal, even when assuming that $\mu$ is compactly supported.
Very recently, strickingly clever  alternatives were considered by Boissard-Le Gouic \cite{blg} 
and by Dereich-Scheutzow-Schottstedt \cite{dss}. Unfortunately, the construction of 
Boissard-Le Gouic, based on iterative trees, was a little too complicated to yield sharp rates.
On the contrary, the method of \cite{dss}, exposed in details in the next section, 
is extremely simple, robust, and leads to the almost optimal results exposed here. 
Some sharp moment estimates were already obtained in \cite{dss} for a limited range of parameters.

\vip

Concerning concentration estimates, only few results are available. Let us mention the work
of Bolley-Guillin-Villani \cite{bgv} and very recently by Boissard \cite{b}, on which we considerably improve. Our assumptions
are often much weaker (the reference measure $\mu$ was often assumed to satisfy some functional
inequalities, which may be difficult to verify and usually include more "structure" than mere integrability conditions) and
$\Pr[\cT_p (\mu_N,\mu)\geq x]$ was estimated only for rather large values of $x$. In particular,
when integrating the concentration estimates of \cite{bgv}, 
one does never find the {\it good} moment estimates, meaning that the scales are not the good ones.

\vip

Moreover, the approach of \cite{dss} is robust enough so that we can also give some good moment 
bounds for the Wasserstein distance between the empirical measure of a Markov 
chain and its invariant distribution (under some conditions). This could be useful for MCMC methods
because our results are non asymptotic.
We can also study very easily some $\rho$-mixing sequences (see Doukhan \cite{douk}), 
for which only very few results exist, see Biau-Devroye-Lugosi \cite{blg}.
Finally, we show on an example how to use Theorem \ref{mainmom} to study some particle systems.
For all these problems, we might also obtain some concentration inequalities, but this would
need further refinements which are out of the scope of the present paper, somewhat already 
technical enough, and left for further works.

\subsection{Plan of the paper}

In the next section, we state some general upper bounds of $\cT_p(\mu,\nu)$,
for any $\mu,\nu \in \cP(\rd)$, essentially taken from \cite{dss}.
Section \ref{prmm} is devoted to the proof of Theorem \ref{mainmom}.
Theorem \ref{maindev} is proved in three steps: in Section \ref{cp} we study
the case where $\mu$ is compactly supported and where $N$ is replaced by a Poisson$(N)$-distributed
random variable,
which yields some  pleasant independance properties. We show how to remove the randomization
in Section \ref{dp}, concluding the case where
$\mu$ is compactly supported. The non compact case is studied in Section \ref{nc}.
The final Section 7 is devoted to dependent random variables: $\rho$-mixing sequences, 
Markov chains and a particular particle system.

\section{Coupling}

The following notion of distance, essentially taken from \cite{dss},  
is the main ingredient of the paper.

\begin{nota}\label{dfdp}
(a) For $\ell\geq 0$, we denote by $\cP_\ell$ the natural partition of $(-1,1]^d$ into $2^{d\ell}$
translations of $(-2^{-\ell},2^{-\ell}]^d$.
For two probability measures $\mu,\nu$ on $(-1,1]^d$ and for $p>0$, we introduce
$$
\cD_p(\mu,\nu):= \frac{2^p-1}{2} \sum_{\ell\geq 1} 2^{-p\ell} \sum_{F \in \cP_\ell} |\mu(F)-\nu(F)|,
$$
which obviously defines a distance on $\cP((-1,1]^d)$, always bounded by $1$.

(b) We introduce $B_0:=(-1,1]^d$ and, for
$n\geq 1$, $B_n:=(-2^n,2^n]^d \setminus (-2^{n-1},2^{n-1}]^d$.
For $\mu \in \cP(\rd)$ and $n\geq 0$, we denote by $\cR_{B_n} \mu$ the probability measure
on $(-1,1]^d$ defined as the image of $\mu|_{B_n}/\mu(B_n)$ by the map $x \mapsto x/2^n$.
For two probability measures $\mu,\nu$ on $\rd$ and for $p>0$, we introduce
$$
\cD_p(\mu,\nu):= \sum_{n\geq 0} 2^{pn} \big( |\mu(B_n)-\nu(B_n)| 
+ (\mu(B_n)\land \nu(B_n)) \cD_p(\cR_{B_n}\mu,\cR_{B_n}\nu) \big).
$$
A little study, using that $\cD_p \leq 1$ on $\cP((-1,1]^d)$, 
shows that this defines a distance on $\cP(\rd)$.
\end{nota}

Having a look at $\cD_p$ in the compact case, one sees that in some sense, 
it measures distance of the two probability measures simultaneously at all the scales.
The optimization procedure can be made for all scales and outperforms the 
approach based on a fixed diameter covering of the state space
(which is more or less the approach of Horowitz-Karandikar \cite{hk}). Moreover one 
sees that the principal control is on $|\pi(F)-\mu(F)|$ which is a quite simple quantity. The 
next results are slightly modified versions of estimates found in \cite{dss},
see \cite[Lemma 2]{dss} for the compact case and \cite[proof of Theorem 3]{dss}
for the non compact case. It contains the crucial remark that $\cD_p$ is an upper bound 
(up to constant) of the Wasserstein distance.

\begin{lem}\label{fonda}
Let $d\geq 1$ and  $p>0$.
For all pairs of probability measures $\mu,\nu$ on $\rd$, $\cT_p(\mu,\nu)\leq \kappa_{p,d}
\cD_p(\mu,\nu)$, with $\kappa_{p,d}:= 2^{p(1+d/2)}(2^p+1)/(2^p-1)$.
\end{lem}

\begin{proof} 
We separate the proof into two steps.

\vip

{\bf Step 1.} We first assume that $\mu$ and $\nu$ are supported in $(-1,1]^d$.
We infer from \cite[Lemma 2]{dss}, in which the conditions
$p\geq 1$ and $d\geq 3$ are clearly not used, that, since the diameter of $(-1,1]^d$ is $2^{1+d/2}$,
$$
\cT_p(\mu,\nu) \leq 2^{p(1+d/2)-1} \sum_{\ell\geq 0} 2^{-p\ell} \sum_{F \in \cP_\ell} \mu(F)
\sum_{C\; child\; of\; F}\left|\frac{\mu(C)}{\mu(F)}-\frac{\nu(C)}{\nu(F)}\right|,
$$
where ``$C$ child of $F$'' means that $C\in\cP_{\ell+1}$ and $C \subset F$. Consequently,
\begin{align*}
\cT_p(\mu,\nu) \leq& 2^{p(1+d/2)-1} \sum_{\ell\geq 0} 2^{-p\ell} \sum_{F \in \cP_\ell} 
\sum_{C\; child\; of\; F} \left(\frac{\nu(C)}{\nu(F)}|\mu(F)-\nu(F)|+|\mu(C)-\nu(C)|\right)\\
\leq& 2^{p(1+d/2)-1} \sum_{\ell\geq 0} 2^{-p\ell} \left( \sum_{F \in \cP_\ell} |\mu(F)-\nu(F)|
+ \sum_{C \in \cP_{\ell+1}} |\mu(C)-\nu(C)|\right)\\
\leq & 2^{p(1+d/2)-1}(1+2^p) \sum_{\ell\geq 1} 2^{-p\ell} \sum_{F \in \cP_\ell} |\mu(F)-\nu(F)|,
\end{align*}
which is nothing but $\kappa_{p,d}\cD_p(\mu,\nu)$. We used that 
$\sum_{F \in \cP_0} |\mu(F)-\nu(F)|=0$.

\vip

{\it In \cite{dss}, Dereich-Scheutzow-Schottstedt use directly the formula with 
the children to study the rate of convergence of empirical measures. 
This leads to some (small) technical complications, and does not seem to improve the estimates.}

\vip

{\it Step 2.} We next consider the general case.
We consider, for each $n\geq 1$, the optimal coupling $\pi_n(dx,dy)$ 
between $\cR_{B_n}\mu$ and $\cR_{B_n}\nu$ for $\cT_p$. We define $\xi_n(dx,dy)$
as the image of $\pi_n$ by the map $(x,y)\mapsto (2^nx,2^ny)$, which clearly belongs
to $\cH(\mu\vert_{B_n}/\mu(B_n),\nu\vert_{B_n}/\nu(B_n))$
and satisfies $\int\!\!\int |x-y|^p \xi_n(dx,dy) = 2^{np} \int\!\!\int |x-y|^p \pi_n(dx,dy)
=2^{np} \cT_p(\cR_{B_n}\mu,\cR_{B_n}\nu)$.

\vip

Next, we introduce $q:=\frac12 \sum_{n\geq 0} |\nu(B_n)-\mu(B_n)|$
and we define
\begin{align*}
\xi(dx,dy)=&\sum_{n\geq 0} (\mu(B_n)\land\nu(B_n))\xi_n(dx,dy) + \frac{\alpha(dx)\beta(dy)}q,
\end{align*}
where
$$
\alpha(dx):= \sum_{n\geq 0}(\mu(B_n)-\nu(B_n))_+ \frac{\mu\vert_{B_n}(dx)}{\mu(B_n)} \hbox{ and }
\beta(dy):= \sum_{n\geq 0}(\nu(B_n)-\mu(B_n))_+ \frac{\nu\vert_{B_n}(dy)}{\nu(B_n)}.
$$
Using that
$$
q=\sum_{n\geq 0}(\nu(B_n)-\mu(B_n))_+=\sum_{n\geq 0}(\mu(B_n)-\nu(B_n))_+
=1-\sum_{n\geq 0} (\nu(B_n)\land\mu(B_n)),
$$
it is easily checked that $\xi \in \cH(\mu,\nu)$.
Furthermore, we have, setting $c_p=1$ if $p\in(0,1]$ and $c_p=2^{p-1}$ if $p>1$,
\begin{align*}
\int\!\!\int |x-y|^p \frac{\alpha(dx)\beta(dy)}q
\leq& \frac 1 {q} \int\!\!\int c_p(|x|^p +|y|^p)
\alpha(dx)\beta(dy)\\
= &c_p\int |x|^p \alpha(dx) +  c_p \int |y|^p \beta(dy)\\
\leq & c_p \sum_{n\geq 0} 2^{pn} [(\mu(B_n)-\nu(B_n))_+ +(\nu(B_n)-\mu(B_n))_+ ]\\
=& c_p \sum_{n\geq 0} 2^{pn} |\mu(B_n)-\nu(B_n)|.
\end{align*}
Recalling that $\int\!\!\int |x-y|^p \xi_n(dx,dy) \leq 2^{np} \cT_p(\cR_{B_n}\mu,\cR_{B_n}\nu)$,
we deduce that
\begin{align*}
\cT_p(\mu,\nu)\leq& \int\!\!\int |x-y|^p \xi(dx,dy)\\
\leq & \sum_{n\geq 0} 2^{np} \left(c_p|\mu(B_n)-\nu(B_n)|+ 
(\mu(B_n)\land\nu(B_n))\cT_p(\cR_{B_n}\mu,\cR_{B_n}\nu)\right).
\end{align*}
We conclude using Step 1 and that $c_p\leq \kappa_{p,d}$.
\end{proof}

When proving the concentration inequalities, which is very technical, it will be good to
break the proof into several steps to separate the difficulties and we will
first treat the compact case.
On the contrary, when dealing with moment estimates, the following formula will be
easier to work with.

\begin{lem}\label{ridic}
Let $p>0$ and $d\geq 1$. There is a constant $C$, depending only on $p,d$, such that
for all $\mu,\nu \in \cP(\rd)$, 
$$
\cD_p(\mu,\nu) \leq C \sum_{n\geq 0} 2^{pn} \sum_{\ell\geq 0} 2^{-p\ell} \sum_{F\in\cP_\ell} |\mu(2^nF\cap B_n)
-\nu(2^n F \cap B_n)|
$$
with the notation $2^n F = \{2^n x\;:\; x\in F\}$.
\end{lem}

\begin{proof}
For all $n\geq 1$, we have  
$|\mu(B_n)-\nu(B_n)|=\sum_{F\in\cP_0}|\mu(2^nF \cap B_n)-\nu(2^n F \cap B_n)|$ and
$(\mu(B_n)\land \nu(B_n)) \cD_p(\cR_{B_n}\mu, \cR_{B_n}\nu)$ is smaller than
\begin{align*}
&(\mu(B_n)\land \nu(B_n)) \cD_p(\cR_{B_n}\mu, \cR_{B_n}\nu)\\
\leq&\mu(B_n)\sum_{\ell\geq 1} 2^{-p\ell}\sum_{F\in\cP_\ell} \left|\frac{\mu(2^n F \cap B_n)}{\mu(B_n)}
- \frac{\nu(2^n F \cap B_n)}{\nu(B_n)}\right|\\
\leq & \sum_{\ell\geq 1}2^{-p\ell} \sum_{F\in\cP_\ell} \left|\mu(2^n F \cap B_n)-\nu(2^n F \cap B_n)\right|
+\left|1-\frac{\mu(B_n)}{\nu(B_n)} \right|\sum_{\ell\geq 1}2^{-p\ell} \sum_{F\in\cP_\ell} \nu(2^n F \cap B_n).
\end{align*}
This last term is smaller than $2^{-p} |\mu(B_n)-\nu(B_n)| /(1-2^{-p})$ and this ends the proof.
\end{proof}

\section{Moment estimates}\label{prmm}

The aim of this section is to give the

\begin{preuve} {\it of Theorem \ref{mainmom}.} We thus assume that $\mu\in \cP(\rd)$ and
that $M_q(\mu)<\infty$ for some $q>p$. By a scaling argument, we may assume that $M_q(\mu)=1$.
This implies that $\mu(B_n)\leq 2^{-q(n-1)}$ for all $n\geq 0$.
By Lemma \ref{fonda}, we have $\cT_p(\mu_N,\mu)\leq \kappa_{p,d} \cD_p(\mu_N,\mu)$, so that
it suffices to study $\E(\cD_p(\mu_N,\mu))$.

\vip

For a Borel subset $A\subset \rd$, since
$N\mu_N(A)$ is Binomial$(N,\mu(A))$-distributed, we have
$$
\E(|\mu_N(A)-\mu(A)|) \leq \min \left\{2\mu(A),\sqrt{\mu(A)/N}\right\}.
$$
Using the Cauchy-Scharz inequality and that $\#(\cP_\ell)=2^{d\ell}$, we deduce that
for all $n\geq 0$, all $\ell \geq 0$,
$$
\sum_{F\in\cP_\ell} \E(|\mu_N(2^nF \cap B_n)-\mu(2^nF \cap B_n)|) 
\leq \min\left\{ 2\mu(B_n), 2^{d\ell/2} (\mu(B_n)/N)^{1/2}\right\}.
$$
Using finally Lemma \ref{ridic} and that $\mu(B_n)\leq 2^{-q(n-1)}$, we find
\begin{align}\label{topcool}
\E(\cD_p(\mu_N,\mu))\leq C \sum_{n\geq 0} 2^{pn} \sum_{\ell\geq 0} 2^{-p\ell} \min\left\{ 
2^{-qn}, 2^{d\ell/2} (2^{-qn}/N)^{1/2}\right\}.
\end{align}

{\bf Step 1.} Here we show that for all $\e\in(0,1)$, all $N\geq 1$,
\begin{align*}
\sum_{\ell\geq 0} 2^{-p\ell} \min\left\{ \e, 2^{d\ell/2} (\e/N)^{1/2}\right\}
\leq& C
\left\{\begin{array}{lll}
\min\{\e,(\e/N)^{1/2}\} & \hbox{if} & p>d/2, \\[+3pt]
\min\{\e,(\e/N)^{1/2}\log(2+\e N)\}& \hbox{if} & p=d/2, \\[+3pt]
\min\{\e,\e (\e N)^{-p/d}\} & \hbox{if} & p\in(0,d/2).
\end{array}\right.
\end{align*}
First of all, the bound by $C\e$ is obvious in all cases (because $p>0$). 
Next, the case $p>d/2$ is immediate.
If $p\leq d/2$, we introduce $\ell_{N,\e}:= \lfloor \log(2+\e N)/(d\log 2)\rfloor$,
for which $2^{d \ell_{N,\e}} \simeq 2+\e N$ and get an upper bound in
$$
(\e /N)^{1/2} \sum_{\ell \leq \ell_{N,\e}} 2^{(d/2-p)\ell} + \e \sum_{\ell \geq \ell_{N,\e}} 2^{-p\ell}.
$$
If $p=d/2$, we find an upper bound in $$(\e /N)^{1/2}\ell_{N,\e}
+ C \e 2^{-p\ell_{N,\e}} \leq C(\e /N)^{1/2}\log(2+\e N) +C \e (1+\e N)^{-1/2}\leq 
C(\e /N)^{1/2}\log(2+\e N)$$ as desired. If $p\in(0,d/2)$, we get an upper bound in
$$
C(\e /N)^{1/2} 2^{(d/2-p)\ell_{N,\e}}+ C \e 2^{-p\ell_{N,\e}} \leq C (\e /N)^{1/2} (2+\e N)^{1/2-p/d}
+ C \e (2+\e N)^{-p/d}.
$$
If $\e N \geq 1$, then $(2+\e N)^{1/2-p/d}\leq (3\e N)^{1/2-p/d}$ and the conclusion follows.
If now $\e N \in (0,1)$, the result is obvious because $\min\{\e,\e (\e N)^{-p/d}\}=\e$.

\vip

{\bf Step 2:} $p>d/2$. By \eqref{topcool} and Step 1 (with $\e=2^{-qn}$), we find
\begin{align*}
\E(\cD_p(\mu_N,\mu))\leq C \sum_{n\geq 0} 2^{pn} \min\left\{2^{-qn},(2^{-qn}/N)^{1/2}\right\}
\leq C \left\{\begin{array}{lll}
N^{-1/2} & \hbox{ if } & q>2p, \\[+3pt]
N^{-(q-p)/q}    & \hbox{ if } & q\in (p,2p).
\end{array}\right.
\end{align*}
Indeed, this is obvious if $q>2p$, while the case $q\in (p,2p)$ requires to 
separate the sum in two parts $n \leq n_N$ and $n>n_N$ with $n_N=\lfloor \log N /(q\log 2)\rfloor$.
This ends the proof when $p>d/2$.
\vip

{\bf Step 3:} $p=d/2$. By \eqref{topcool} and Step 1 (with $\e=2^{-qn}$), we find
\begin{align*}
\E(\cD_p(\mu_N,\mu))\leq C \sum_{n\geq 0} 2^{pn} \min\left\{2^{-qn},(2^{-qn}/N)^{1/2}\log(2+2^{-qn} N)
\right\}.
\end{align*}

If $q>2p$, we immediately get a bound in
\begin{align*}
\E(\cD_p(\mu_N,\mu))\leq
C \sum_{n\geq 0} 2^{(p-q/2)n} N^{-1/2} \log(2+N) \leq C \log(2+N) N^{-1/2},
\end{align*}
which ends the proof (when $p=d/2$ and $q>2p$).

\vip

If $q\in (p,2p)$, we easily obtain, using that $\log(2+x)\leq 2\log x$ for all $x\geq 2$,
an upper bound in 
\begin{align*}
\E(\cD_p(\mu_N,\mu))\leq
& C \sum_{n\geq 0} \indiq_{\{N < 2.2^{nq}\}} 2^{(p-q)n} 
+ C \sum_{n\geq 0} \indiq_{\{N \geq 2.2^{nq}\}}2^{(p-q/2)n} N^{-1/2} \log(N2^{-nq})\\
\leq& C N^{-(q-p)/q} + CN^{-1/2}  \sum_{n=0}^{n_N}  2^{(p-q/2)n} (\log N - nq \log 2)\\
=:&CN^{-(q-p)/q} +CN^{-1/2} K_N,
\end{align*}
where $n_N=\lfloor \log (N/2) / (q \log 2)\rfloor$. A tedious exact computation shows that
\begin{align*}
K_N=&\log N \frac{2^{(p-q/2)(n_N+1)}-1}{2^{(p-q/2)}-1} \\
&- q\log 2 \left[(n_N+1)\frac{2^{(p-q/2)(n_N+1)}-1}{2^{(p-q/2)}-1} + \frac{n_N+1}{2^{(p-q/2)}-1}
- \frac{2^{(p-q/2)(n_N+2)}-2^{(p-q/2)}}{(2^{(p-q/2)}-1)^2}   \right].
\end{align*}
Using that the contribution of the middle term of the second line is negative and 
the inequality $\log N - (n_N+1)q\log 2 \leq \log 2$ (because  
$(n_N+1)q\log 2 \geq \log (N/2)$), we find
$$
K_N \leq C 2^{(p-q/2)n_N}
\leq C N^{p/q-1/2}.
$$
We finally have checked that $\E(\cD_p(\mu_N,\mu))\leq CN^{-(q-p)/q}+CN^{-1/2}N^{p/q-1/2}
\leq C N^{-(q-p)/q} $, which ends the proof when  $p=d/2$.

\vip

{\bf Step 4:} $p\in (0,d/2)$. We then have, by \eqref{topcool} and Step 1,
\begin{align*}
\E(\cD_p(\mu_N,\mu)) \leq& C\sum_{n\geq 0} 2^{pn} \min\left\{2^{-qn},
2^{-qn(1-p/d)} N^{-p/d}\right\}.
\end{align*}

If $q>dp/(d-p)$, which implies that
$q(1-p/d)>p$, we immediately get an upper bound by $C N^{-p/d}$, which
ends the proof when $p<d/2$ and $q>dp/(d-p)$

\vip

If finally $q\in (p,dp/(d-p))$, we
separate the sum in two parts $n \leq n_N$ and $n>n_N$ with $n_N=\lfloor \log N /(q\log 2)\rfloor$
and we find a bound in $C N^{-(q-p)/q}$ as desired.
\end{preuve}


\section{Concentration inequalities in the compact poissonized case}\label{cp}

It is technically advantageous to first consider the case where the size of the
sampling is Poisson distributed, which implies some independence properties. 
If we replace $N$ (large) by a Poisson$(N)$-distributed random variable,
this should not change much the problem, because a Poisson$(N)$-distributed random variable
is close to $N$ with high probability. 

\begin{nota}\label{fg}
We introduce the functions $f$ and $g$ defined on $(0,\infty)$ by
$$
f(x)=(1+x)\log(1+x)-x \quad \hbox{and} \quad g(x)=(x \log x - x + 1) \indiq_{\{x\geq 1\}}.
$$
Observe that $f$ is increasing, nonnegative, equivalent to $x^2$ at $0$ and to
$x\log x$ at infinity. The function $g$ is positive and increasing
on $(1,\infty)$.
\end{nota}

The goal of this section is to check the following.

\begin{prop}\label{devpoisscomp}
Assume that $\mu$ is supported in $(-1,1]^d$. Let 
$\Pi_N$ be a Poisson measure on $\rd$ with intensity measure $N\mu$ and introduce
the associated empirical measure $\Psi_N=(\Pi_N(\rd))^{-1}\Pi_N$. Let 
$p\geq 1$ and $d\geq 1$.
There are some positive constants $C,c$ (depending only on $d,p$) 
such that for all $N\geq 1$, all $x\in (0,\infty)$,
$$
\P\left(\Pi_N(\rd)\cD_p(\Psi_N,\mu)\geq N x\right)
\leq C \left\{\begin{array}{ll}
\exp(-N f(c x))&\hbox{if $p>d/2$,}\\[+3pt]
\exp\left(- N f(c x/\log(2+1/x))\right) &\hbox{if $p=d/2$,}\\[+3pt]
\exp\left(- Nf(c x)\right) + \exp\left(-c N x^{d/p}\right)&\hbox{if $p\in[1,d/2)$.}
\end{array}\right.
$$
\end{prop}

We start with some easy and well-known concentration inequalities for the Poisson distribution.

\begin{lem}\label{devpoiss}
For $\lambda>0$ and $X$ a Poisson$(\lambda)$-distributed random variable, we have

(a) $E(\exp(\theta X)) = \exp(\lambda (e^\theta -1))$ for all $\theta \in \rr$;

(b) $E(\exp(\theta|X-\lambda|)) \leq 2 \exp(\lambda(e^\theta-1-\theta) )$ for all $\theta>0$;

(c) $\P(X>\lambda x) \leq \exp(-\lambda g(x))$ for all $x>0$;

(d) $\P(|X-\lambda|>\lambda x) \leq 2 \exp(-\lambda f(x))$ for all $x>0$;

(e) $\P(X>\lambda x) \leq \lambda$ for all $x>0$.
\end{lem}

\begin{proof}
Point (a) is straightforward. For point (b), write
$E(\exp(\theta|X-\lambda|))\leq e^{\theta \lambda}\E(\exp(-\theta X))+  e^{-\theta \lambda}
\E(\exp(\theta X))$, use (a) and that $\lambda(e^{-\theta}-1+\theta)\leq \lambda(e^\theta-1-\theta)$.
For point (c), write $\P(X>\lambda x)\leq e^{- \theta \lambda x}\E[\exp(\theta X)]$,
use (a) and optimize in $\theta$. Use the same scheme to deduce (d) from (b).
Finally, for $x>0$, $\P(X>\lambda x)\leq\P(X>0)=1-e^{-\lambda}\leq \lambda$.
\end{proof}

We can now give the

\begin{preuve} {\it of Proposition \ref{devpoisscomp}.}
We fix $x>0$ for the whole proof.
Recalling Notation \ref{dfdp}-(a), we have
\begin{align*}
\Pi_N(\rd)&\cD_p(\Psi_N,\mu) = C\sum_{\ell \geq 1} 2^{-p\ell} \sum_{F\in \cP_\ell}
|\Pi_N(F) - \Pi_N(\rd)\mu(F)| \\
\leq &  C|\Pi_N(\rd)-N| + C\sum_{\ell \geq 1} 2^{-p\ell} \sum_{F\in \cP_\ell} |\Pi_N(F) - N\mu(F)|\\
\leq & C|\Pi_N(\rd)-N| + C (N+\Pi_N(\rd))2^{-p \ell_0} + 
C\sum_{\ell=1}^{\ell_0} 2^{-p\ell} \sum_{F\in \cP_\ell} |\Pi_N(F) - N\mu(F)| 
\end{align*}
for any choice of $\ell_0\in\nn$. We will choose $\ell_0$ later, depending on
the value of $x$.
For any nonnegative family $r_\ell$ such that $\sum_1^{\ell_0} r_\ell\leq 1$, we thus have
\begin{align*}
\e(N,x):=&\P\left(\Pi_N(\rd)\cD_p(\Psi_N,\mu)\geq Nx\right) \\
\leq & \P\left(|\Pi_N(\rd)-N| \geq c Nx\right) 
+ \P\left(\Pi_N(\rd)\geq N(c x 2^{p \ell_0} -1)\right)\\
& +\sum_{\ell=1}^{\ell_0} \P\left(\sum_{F\in \cP_\ell} |\Pi_N(F) - N\mu(F)| 
\geq c N x 2^{p\ell}r_\ell \right).
\end{align*}
By Lemma \ref{devpoiss}-(c)-(d), since $\Pi_N(\rd)$ is Poisson$(N)$-distributed,
$\P(\Pi_N(\rd)\geq N(c x 2^{p \ell_0} -1))\leq \exp(-N g(c x 2^{p \ell_0} -1))$ and 
$\P(|\Pi_N(\rd)-N| \geq c Nx) \leq 2 \exp(-Nf(cx))$. Next, using that
the family $(\Pi_N(F))_{F\in\cP_\ell}$ is independent, with $\Pi_N(F)$ Poisson$(N\mu(F))$-distributed,
we use Lemma \ref{devpoiss}-(a) and that $\#(\cP_\ell)=2^{\ell d}$ to obtain, for any $\theta>0$,
$$
\E\left(\exp\left(\theta\sum_{F\in \cP_\ell} |\Pi_N(F) - N\mu(F)| \right)\right) 
\leq \prod_{F\in \cP_\ell} 2e^{N\mu(F)(e^\theta-\theta-1)}
\leq 2^{2^{d\ell}} e^{N(e^\theta-\theta-1)}.
$$
Hence
\begin{align*}
\P\left(\sum_{F\in \cP_\ell} |\Pi_N(F) - N\mu(F)| \geq c N x2^{p\ell} r_\ell \right) \leq &
\exp\left(-c \theta N x2^{p\ell} r_\ell\right) 2^{2^{d\ell}} \exp\left(N(e^\theta-\theta-1)\right).
\end{align*}
Choosing $\theta= \log(1+c x2^{p\ell} r_\ell  )$, we find
\begin{align*}
\P\left(\sum_{F\in \cP_\ell} |\Pi_N(F) - N\mu(F)| \geq c N x2^{p\ell} r_\ell \right) \leq 
&2^{2^{d\ell}} \exp(-Nf(c x 2^{p\ell}r_\ell )).
\end{align*}
We have checked that 
\begin{align*}
\e(N,x)\leq 2 \exp(-Nf(cx)) + \exp(-N g(c x 2^{p \ell_0} -1)) + 
\sum_{\ell=1}^{\ell_0}2^{2^{d\ell}} \exp(-Nf(c x 2^{p\ell}r_\ell )).
\end{align*}
At this point, the value of $c>0$ is not allowed to vary anymore.
We introduce some other positive constants $a$ whose value may change from line to line.
\vip

{\bf Case 1: $cx>2$.} Then we choose $\ell_0=1$ and $r_1=1$. We have
$cx2^{p\ell_0}-1 = 2^pcx -1 \geq (2^p-1)cx+1$ whence  $g(c x 2^{p \ell_0} -1)\geq g((2^p-1)cx+1)
=f((2^p-1)cx)$. We also have $\sum_{\ell=1}^{\ell_0}2^{2^{d\ell}} \exp(-Nf(c x 2^{p\ell}r_\ell ))
=2^{2^d}\exp(-Nf(2^pc x))$. We finally get 
$\e(N,x) \leq C \exp(-Nf(ax))$, which proves the statement (in the three cases, when $cx>2$).

\vip

{\bf Case 2: $cx\leq 2$.} We choose $\ell_0$ so that $(1+2/(cx))\leq 2^{p\ell_0}\leq 2^p (1+2/(cx))$,
i.e.
$$
\ell_0:=\lfloor \log(1+2/(cx))/(p\log 2) \rfloor+1.
$$ 
This implies that
$c x 2^{p \ell_0} \geq 2 + cx$. Hence $g(c x 2^{p \ell_0} -1)\geq g(1+cx)=f(cx)$.
Furthermore, we have $c x 2^{p \ell}r_\ell \leq c x 2^{p \ell_0}\leq 2^p( 2+cx)\leq 2^{p+2}$
for all $\ell \leq \ell_0$, whence $f(c x 2^{p\ell}r_\ell) \geq a x^2 2^{2p\ell} r_\ell^2$
(because $f(x)\geq a x^2$ for all $x\in [0,2^{p+2}]$).
We thus end up with (we use that $2^{2^{d\ell}}\leq \exp(2^{d\ell})$)
$$
\e(N,x) \leq 3\exp(-Nf(cx))+ \sum_{\ell=1}^{\ell_0}\exp(2^{d\ell} - N a  x^2 2^{2p\ell}r^2_\ell).
$$
Now the value of $a>0$ is not allowed to vary anymore, and we introduce $a'>0$,
whose value may change from line to line.

\vip

{\it Case 1.1: $p>d/2$.}
We take $r_\ell:=(1-2^{-\eta})2^{- \eta\ell}$ for some $\eta>0$ such that
$2(p-\eta)>d$. If $Nx^2\geq 1$, we easily get 
\begin{align*}
\e(N,x) \leq& 3\exp(-Nf(cx))+ \sum_{\ell=1}^{\ell_0}\exp(2^{d\ell} - N a' x^2 2^{2(p-\eta)\ell})\\
\leq&  3\exp(-Nf(cx)) + C \exp(-a'Nx^2) \\
\leq& C \exp(-Nf(a'x)).
\end{align*}
The last inequality uses that $y^2\geq f(y)$ for all $y>0$. If finally $Nx^2\leq 1$,
we obviously have
\begin{align*}
\e(N,x)\leq 1 \leq \exp(1-Nx^2)\leq C \exp(-Nx^2) \leq C \exp(-Nf(x)).
\end{align*}
We thus always have 
$\e(N,x)\leq C \exp(-Nf(a'x))$ as desired.

\vip

{\it Case 2.2:  $p=d/2$.}
We choose $r_\ell:=1/\ell_0$. Thus, if $a N(x/\ell_0)^2 \geq 2$, we easily find
\begin{align*}
\e(N,x) \leq& 3\exp(-Nf(cx))+ \sum_{\ell=1}^{\ell_0}\exp(2^{d\ell}(1 - a N (x/\ell_0)^2)\\
\leq&  3\exp(-Nf(cx)) + C \exp(-a'N(x/\ell_0)^2)\\
\leq & 3\exp(-Nf(cx)) + C \exp(-Nf(a'x/\ell_0))\\
\leq & C \exp(-Nf(a'x/\ell_0))
\end{align*}
because $\ell_0 \geq 1$ and $f$ is increasing.
If now $a N(x/\ell_0)^2 < 2$, we just write
\begin{align*}
\e(N,x)\leq 1 \leq \exp(2-aN(x/\ell_0)^2)\leq C \exp(-aN(x/\ell_0)^2) 
\leq C \exp(-Nf(ax/\ell_0)).
\end{align*}
We thus always have $\e(N,x)\leq C \exp(-Nf(a'x/\ell_0))$. 
Using that $\ell_0 \leq C \log(2+1/x)$, we immediately conclude that
$\e(N,x)\leq C \exp(-Nf(a'x/\log(2+1/x)))$ as desired.

\vip

{\it Case 2.3: $p\in[1,d/2)$.} We choose $r_\ell:= \kappa 2^{(d/2-p)(\ell-\ell_0)}$ with 
$\kappa=1/(1-2^{p-d/2})$.
For all $\ell\leq \ell_0$, 
\begin{align*}
2^{dl}-a Nx^2 2^{2p\ell}r_\ell^2 =& -a \kappa^2 N x^{d/p} 2^{2p\ell}\left[2^{(d-2p)(\ell-\ell_0)} x^{2-d/p}
- 2^{(d-2p)\ell}/(N  a x^{d/p})\right]\\
\leq & -a\kappa^2Nx^{d/p}2^{2p\ell}\left[b 2^{(d-2p)\ell}- 2^{(d-2p)\ell}/(N a \kappa^2 x^{d/p})\right]
\end{align*}
where the constant $b>0$ is such that $2^{-(d-2p)\ell_0}\geq b x^{d/p-2}$ (the existence of
$b$ is easily checked).
Hence if $N a \kappa^2 x^{d/p}\geq 2/b$, we find
\begin{align*}
2^{dl}-a Nx^22^{2p\ell}r_\ell^2 \leq & -a b \kappa^2 Nx^{d/p}2^{d\ell} b/2
\end{align*}
and thus, still using that $N x^{d/p}\geq 2/(ab\kappa^2)$,
$$
\sum_{\ell=1}^{\ell_0}\exp(2^{d\ell} - N c^2 x^2 2^{2p\ell}r^2_\ell) \leq C\exp(-a' N x^{d/p}).
$$
Consequently, we have $\e(N,x) \leq 3\exp(-Nf(cx)) +C\exp(-a' N x^{d/p})$
if  $N a \kappa^2 x^{d/p}\geq 2/b$. As usual, the case where  $N a \kappa^2 x^{d/p}\leq 2/b$
is trivial, since then
\begin{align*}
\e(N,x)\leq 1 \leq \exp(2/b -N a \kappa^2 x^{d/p}) \leq C \exp(-a' N x^{d/p}).
\end{align*}
This ends the proof.
\end{preuve}


\section{Depoissonization in the compact case}\label{dp}

We next check the following compact version of Theorem \ref{maindev}.

\begin{prop}\label{devcomp}
Assume that $\mu$ is supported in $(-1,1]^d$. Let $p>0$ and $d\geq 1$ be fixed.
There are some positive constants $C$ and $c$ 
(depending only on $p,d$) such that for all $N\geq 1$, all 
$x\in (0,\infty)$,
$$
\P\left[\cD_p(\mu_N,\mu)\geq x\right] \leq \indiq_{\{x\leq 1\}}C \left\{\begin{array}{ll}
\exp(-c N x^2) & \hbox{if $p>d/2$};\\[+3pt]
\exp\left(- c N (x/\log(2+1/x))^2\right)& \hbox{if $p=d/2$};\\[+3pt]
\exp\left(-c N x^{d/p}\right)& \hbox{if $p \in (0,d/2)$}.
\end{array}\right.
$$
\end{prop}

We will need the following easy remark.

\begin{lem}\label{pasdur}
For all $N\geq 1$, 
for $X$ Poisson$(N)$-distributed, for all
$k\in \{0,\dots, \lfloor \sqrt N \rfloor\}$, 
$$
\P[X=N+k] \geq \kappa_0 N^{-1/2} \hbox{ where } \kappa_0= e^{-2}/\sqrt 2.
$$
\end{lem}

\begin{proof}
By Perrin \cite{p}, we have $N! \leq e \sqrt N (N/e)^N$.
Thus 
\begin{align*}
\P[X=N+k] = & e^{-N} \frac{N^{N+k}}{(N+k)!} 
\geq e^{-N-1} \frac{N^{N+k}}{ \sqrt{N+k} ((N+k)/e)^{N+k}}
\geq \frac1{\sqrt{2N}} \left(\frac{N}{N+k} \right)^{N+k}e^{k-1}.
\end{align*}
Since $\log(1+x)\leq x$ on $(0,1)$, we have
$((N+k)/N)^{N+k}\leq \exp(k+k^2/N)\leq \exp(k+1)$, so that
$\P[X=N+k] \geq e^{-2}/\sqrt{2N}$.
\end{proof}

\begin{preuve} {\it of Proposition \ref{devcomp}.}
The probability indeed vanishes if $x>1$, since $\cD_p$ is smaller than $1$
when restricted to probability measures on $(-1,1]^d$.

\vip

{\bf Step 1.}
We introduce a Poisson measure $\Pi_N$ on
$\rd$ with intensity measure $N\mu$ and the associated empirical measure
$\Psi_N=\Pi_N/\Pi_N(\rd)$. Conditionally on $\{\Pi_N(\rd)=n\}$, $\Psi_N$ has the same law as
$\mu_n$ (the empirical measure of $n$ i.i.d. random variables with law $\mu$). Consequently,
$$
\P\left[\Pi_N(\rd)\cD_p(\Psi_N,\mu) \geq Nx\right]
=\sum_{n\geq 0} \P[\Pi_N(\rd)=n]\P\left[n\cD_p(\mu_n,\mu)\geq N x\right].
$$
By Lemma \ref{pasdur} (since $\Pi_N(\rd)$ is Poisson$(N)$-distributed),
$$
\frac 1 {\sqrt N}\sum_{k=0}^{\lfloor \sqrt N \rfloor} \P\left[(N+k)\cD_p(\mu_{N+k},\mu)\geq N x\right] 
\leq \kappa_0^{-1} \P\left[\Pi_N(\rd)\cD_p(\Psi_N,\mu) \geq Nx\right],
$$
which of course implies that (for all $N\geq 1$, all $x>0$),
\begin{align*}
\frac 1 {\sqrt N}\sum_{k=0}^{\lfloor \sqrt N \rfloor} \P\left[\cD_p(\mu_{N+k},\mu)\geq x\right] 
\leq \kappa_0^{-1} \P\left[\Pi_N(\rd)\cD_p(\Psi_N,\mu) \geq Nx\right].
\end{align*}

{\bf Step 2.} Here we prove that there is a constant $A>0$ such that for any $N\geq 1$, any 
$k \in \{0,\dots, \lfloor \sqrt N \rfloor\}$, any $x > A N^{-1/2}$,
$$
\P\left[\cD_p(\mu_{N},\mu)\geq x\right] \leq \P\left[\cD_p(\mu_{N+k},\mu)\geq x/2\right].
$$
Build $\mu_n$ for all values of $n\geq 1$ with the same i.i.d. family of $\mu$-distributed 
random variables $(X_k)_{k\geq 1}$. Then a.s., 
$$
|\mu_{N+k}-\mu_N|_{TV}\leq \left|\frac{k}{N(N+k)}\sum_1^N \delta_{X_j}\right|_{TV}
+ \left|\frac 1{N+k}\sum_{N+1}^{N+k}\delta_{X_j}\right|_{TV}\leq \frac{k}{N+k}\leq \frac{1}{\sqrt N}.
$$
This obviously implies (recall Notation \ref{dfdp}-(a))
that $\cD_p(\mu_{N},\mu_{N+k}) \leq C N^{-1/2}$ a.s. (where $C$ depends only on $p$).
By the triangular inequality, 
$\cD_p(\mu_N,\mu) \leq \cD_p(\mu_{N+k},\mu) + C N^{-1/2}$, whence
$$
\P\left[\cD_p(\mu_{N},\mu)\geq x\right] \leq \P\left[\cD_p(\mu_{N+k},\mu)\geq 
x - CN^{-1/2} \right] 
\leq \P\left[\cD_p(\mu_{N+k},\mu)\geq x/2\right]
$$
if $x- CN^{-1/2}\geq x/2$, i.e. $x\geq 2C N^{-1/2}$.

\vip

{\bf Step 3.} Gathering Steps 1 and 2, we deduce that for all $N\geq 1$, all $x>AN^{-1/2}$,
$$
\P\left[\cD_p(\mu_{N},\mu)\geq x\right] \leq \frac 1 {\sqrt N}
\sum_{k=0}^{\lfloor \sqrt N \rfloor} \P\left[\cD_p(\mu_{N+k},\mu)\geq x/2\right] 
\leq C  \P\left[\Pi_N(\rd)\cD_p(\Psi_N,\mu) \geq Nx/2\right].
$$
We next apply Proposition \ref{devpoisscomp}. Observing that, for 
$x \in (0,1]$,

(i) $\exp(-Nf(cx/2))\leq \exp(-cNx^2)$ (case $p>d/2$),

(ii) $\exp(-Nf(cx/2\log(2+2/x)))\leq \exp(-cN(x/\log(2+1/x)^2)$ (case $p=d/2$),

(iii) $\exp(-Nf(cx/2)) + \exp(c N (x/2)^{d/p})\leq \exp(-cNx^{d/p})$ (case $p\in (0,d/2)$)

\noindent 
concludes the proof when $x>AN^{-1/2}$. But the other case is trivial, because for
$x \leq A N^{-1/2}$,
$$
\P[\cD_p(\mu_{N},\mu)\geq x]\leq 1 \leq \exp(A^2-Nx^2)\leq C\exp(-Nx^2),
$$ 
which is also smaller than $C\exp(-N(x/\log(2+1/x))^2)$ and than $C\exp(-Nx^{d/p})$ (if $d>2p$).
\end{preuve}


\section{Concentration inequalities in the non compact case}\label{nc}

Here we conclude the proof of Theorem \ref{maindev}.
We will need some concentration estimates for the Binomial distribution.

\begin{lem}\label{bin}
Let $X$ be Binomial$(N,p)$-distributed. Recall that $f$ was defined in Notation \ref{fg}.

(a) $\P[|X-Np|\geq N p z]\leq (\indiq_{\{p(1+z) \leq 1\}}+\indiq_{\{z \leq 1\}}) \exp(-Npf(z))$ 
for all $z>0$.

(b) $\P[|X-Np|\geq N p z]\leq Np$ for all $z>1$.

(c) $\E(\exp(-\theta X))=(1-p+pe^{-\theta})^N \leq \exp(-N p (1-e^{-\theta}))$ for $\theta>0$.
\end{lem}

\begin{proof}
Point (c) is straightforward. Point (b) follows from the fact that for $z>1$,
$\P[|X-Np|\geq N p z]=\P[X\geq Np(1+z)]\leq \P[X\ne 0]=1-(1-p)^N \leq pN$. For point (a),
we use Bennett's inequality \cite{ben}, see Devroye-Lugosi
\cite[Exercise 2.2 page 11]{dl}, together with the obvious facts that
$\P[X-Np\geq N p z]=0$ if $p(1+z)>1$ and  $\P[X-Np\leq -N p z]=0$ if $z>1$.
The following elementary tedious computations also works: 
write  $\P[|X-Np|\geq N p z] = \P(X\geq Np(1+z)) + \P(N-X\geq N(1-p+zp))
=:\Delta(p,z)+\Delta(1-p,zp/(1-p))$, observe that $N-X\sim$ Binomial$(N,1-p)$.
Use that $\Delta(p,z)\leq  \indiq_{\{p(1+z)\leq 1\}}\exp(-\theta N p(1+z))(1-p+pe^\theta)^N$ 
and choose
$\theta=\log ((1-p)(1+z)/(1-p-pz))$, this gives $\Delta(p,z)\leq \indiq_{\{p(1+z)\leq 1\}}
\exp(-N[p(1+z)\log (1+z)+(1-p-pz)\log((1-p-pz)/(1-p)) ] )$. A tedious study shows that
$\Delta(p,z)\leq \indiq_{\{p(1+z)\leq 1\}} \exp(-Npf(z))$ and that
$\Delta(1-p,zp/(1-p))\leq \indiq_{\{z \leq 1\}}\exp(-Npf(z))$.
\end{proof}

We next estimate the first term when computing
$\cD_p(\mu_N,\mu)$.

\begin{lem}\label{znpex}
Let $\mu \in \cP(\rd)$ and $p>0$. Assume \eqref{asexfort}, \eqref{asexfaible} or \eqref{asmom}.
Recall Notation \ref{dfdp} and put $Z_N^p:=\sum_{n\geq 0}2^{pn}|\mu_N(B_n)-\mu(B_n)|$. 
Let $x_0$ be fixed.
For all $x>0$, 
\begin{align*}
\P[Z_N^p& \geq x ] \leq C \exp(-c N x^2)\indiq_{\{x\leq x_0\}}\\
&+ C\left\{\begin{array}{ll}
\exp(-cNx^{\alpha/p})\indiq_{\{x> x_0\}} & \hbox{under \eqref{asexfort}}, \\[+3pt]
\exp(-c(Nx)^{(\alpha-\e)/p})\indiq_{\{x\leq x_0\}}+\exp(-c (Nx)^{\alpha/p})\indiq_{\{x>x_0\}} 
& \hbox{$\forall \;\e\in(0,\alpha)$ under \eqref{asexfaible}},\\[+3pt]
N (Nx)^{-(q-\e)/p}& \hbox{$\forall\;\e\in(0,q)$ under \eqref{asmom}}.\end{array}\right.
\end{align*}
\end{lem}

\begin{proof}
Under \eqref{asexfort} or \eqref{asexfaible}, we assume that $\gamma=1$ without loss of generality 
(by scaling), whence 
$\cE_{\alpha,1}(\mu)<\infty$ and thus $\mu(B_n)\leq C e^{-2^{(n-1)\alpha}}$ for all $n\geq 0$.
Under \eqref{asmom}, we have $\mu(B_n)\leq C 2^{-qn}$ for all $n\geq 0$.
For $\eta>0$ to be chosen later 
(observe that $\sum_{n\geq 0} (1-2^{-\eta})2^{-\eta n}=1$), putting $c:=1-2^{-\eta}$
and $z_n:= c x 2^{-(p+\eta)n}/\mu(B_n)$,
\begin{align*}
\P\left(Z_N^p\geq x \right)
\leq& \left(\sum_{n\geq 0} \indiq_{\{z_n \leq 2\}}
\P\left[ |N\mu_N(B_n)-N\mu(B_n)| \geq N \mu(B_n) z_n\right]\right)\land 1 \\
&+ \left(\sum_{n\geq 0}  \indiq_{\{z_n > 2\}}
\P\left[ |N\mu_N(B_n)-N\mu(B_n)| \geq N \mu(B_n) z_n\right]\right)\land 1 \\
=:&  \left(\sum_{n\geq 0} I_n(N,x)\right)\land 1 +  \left(\sum_{n\geq 0} J_n(N,x)\right)\land 1.
\end{align*}
From now on, the value of $c>0$ is not allowed to vary anymore.
We introduce another positive constant $a>0$ whose value may change from line to line.

\vip

{\bf Step 1: bound of $I_n$.}
Here we show that under \eqref{asmom} (which is of course implied
by \eqref{asexfort} or \eqref{asexfaible}), if $\eta \in (0,q/2-p)$, there is $A_0>0$ such that
$$
\sum_{n\geq 0} I_n(N,x) \leq C\exp(-a N x^2)\indiq_{\{x\leq A_0\}} \quad \hbox{if $N x^2 \geq 1$.}
$$
This will obviously imply that for all $N\geq 1$, all $x>0$,
$$
\left(\sum_{n\geq 0} I_n(N,x)\right)\land 1 \leq C \exp(-a N x^2)\indiq_{\{x\leq A_0\}}.
$$

First, $\sum_{n\geq 0} I_n(N,x)=0$ if $z_n>2$ for all $n\geq 0$. Recalling that 
$\mu(B_n)\leq C2^{-qn}$, this is the case if
$x\geq (2C/c)\sup_{n\geq 0}2^{(p+\eta-q)n}=(2C/c):=A_0$.
Next, since $N\mu_N(B_n)\sim$ Binomial$(N,\mu(B_n))$, Lemma \ref{bin}-(a) leads us to
$$
I_n(N,x) \leq 2 \indiq_{\{z_n\leq 2\}} \exp(-N\mu(B_n)f(z_n))
\leq 2 \exp(-N \mu(B_n)z_n^2/4)),
$$
because $f(x) \geq x^2/4$ for $x\in [0,2]$. Since finally
$\mu(B_n) z_n^2/4 \geq a x^2 2^{(q-2p-2\eta)n}$, we easily conclude,
since $q-2p-2\eta>0$ and since $Nx^2 \geq 1$, that
$$
\sum_{n\geq 0} I_n(N,x) \leq C \sum_{n\geq 0} \exp(-a N x^22^{(q-2p-2\eta)n})\indiq_{\{x\leq A_0\}} 
\leq C \exp(-a N x^2)\indiq_{\{x\leq A_0\}}.
$$
\vip

{\bf Step 2: bound of $J_n$ under \eqref{asexfort} or \eqref{asexfaible} when $x\leq A$.} 
Here we fix $A>0$ and prove that if $\eta>0$ is small enough, for all $x\in (0,A]$ such that
$Nx^2\geq 1$,
\begin{align*}
\sum_{n\geq 0} J_n(N,x) \leq & C  \left\{
\begin{array}{ll}
\exp(-a N x^2)  & \hbox{under \eqref{asexfort},}\\[+3pt]
\exp(-a N x^2)+\exp(- a (N x)^{(\alpha-\e)/p}) & \hbox{$\forall\;\e\in(0,\alpha)$ 
under \eqref{asexfaible}.}
\end{array}\right.
\end{align*}
This will imply, as usual, that for all $N\geq 1$, all $x>0$,
\begin{align*}
\left(\sum_{n\geq 0} J_n(N,x)\right)\land 1 \leq & C  \left\{
\begin{array}{ll}
\exp(-a N x^2)& \hbox{under \eqref{asexfort},}\\[+3pt]
\exp(-a N x^2)+\exp(- a (N x)^{(\alpha-\e)/p}) & \hbox{$\forall\;\e\in(0,\alpha)$ 
under \eqref{asexfaible}.}\\[+3pt]
\end{array}\right.
\end{align*}

By Lemma \ref{bin}-(a)-(b) (since $z_n> 2$ implies 
$\indiq_{\{\mu(B_n)(1+z_n)\leq 1\}}+\indiq_{\{z_n\leq 1\}} \leq \indiq_{\{z_n\leq 1/\mu(B_n)\}}$),
\begin{align*}
J_n(N,x)\leq & \indiq_{\{2<z_n\leq 1/\mu(B_n) \}} 
\min\left\{\exp(-N\mu(B_n)f(z_n)),N\mu(B_n) \right\} \\
\leq & \indiq_{\{z_n \mu(B_n)\leq 1\}}
\min \left\{\exp\left(-a N \mu(B_n) z_n \log [2\lor z_n]\right),  N\mu(B_n)\right\}
\end{align*}
because $f(y)\geq a y \log y \geq a y \log [2\lor y]$ for $y> 2$.
Since $\mu(B_n) \leq Ce^{-2^{(n-1)\alpha}}$, we get
$$
J_n(N,x)\leq C 
\min \{\exp(-aN x 2^{-(p+\eta)n} \log [2\lor( a x2^{-(p+\eta)n}e^{2^{(n-1)\alpha}})]),Ne^{2^{-(n-1)\alpha}}\}.
$$
A straightforward computation shows that there is a 
constant $K$ such that
for $n \geq n_1:=\lfloor K(1+\log \log (K/x))\rfloor$,  
we have 
$\log (a x2^{-(p+\eta)n}e^{2^{(n-1)\alpha}})\geq 2^{(n-1)\alpha}/2$. 
Consequently,
\begin{align*}
\sum_{n\geq 0}J_n(N,x) \leq& C n_1 \exp(-aN x2^{-(p+\eta)n_1}) + 
C \sum_{n> n_1} \min\left\{\exp(-aN x 2^{(\alpha-p-\eta)n}), e^{-2^{(n-1)\alpha}}   \right\}\\
=& C J^1(N,x)+C J^2(N,x).
\end{align*}
We first show that $J^1(N,x)\leq Ce^{-aNx^2}$ (here we actually could get something much better).
First, since
$n_1=\lfloor K+K\log \log (K/x)\rfloor$ and $x \in [0,A]$, we clearly have e.g.
$x2^{-(p+\eta)n_1} \geq a x^{3/2}$. Next, $Nx^2\geq 1$ implies that 
$1/x \leq (Nx^{3/2})^2$. Thus
$$
J^1(N,x) \leq C(1+\log\log(C(Nx^{3/2})^2))\exp(-aNx^{3/2}) \leq 
C\exp(-aNx^{3/2}) \leq \exp(-aNx^2).
$$
We now treat $J^2(N,x)$.

\vip

{\it Step 2.1.} Under \eqref{asexfort}, we immediately get, if $\eta \in (0, \alpha-p)$
(recall that $x\in[0,A]$),
$$
J^2(N,x)\leq \sum_{n\geq 0} \exp(-aN x 2^{(\alpha-p-\eta)n})\leq C \exp(-aNx) \leq C\exp(-aNx^2),
$$
where we used that $x \leq A$ and $Nx^2\geq 1$ (whence $Nx\geq 1/A$).

\vip

{\it Step 2.2.} Under \eqref{asexfaible}, we first write
\begin{align*}
J^2(N,x) \leq &\sum_{n\geq 0} \min\left\{\exp(-aN x 2^{(\alpha-p-\eta)n}),e^{-2^{(n-1)\alpha}}\right\}
\leq  n_2 \exp(-cNx 2^{(\alpha-p-\eta)n_2})+ Ne^{-2^{(n_2-1)\alpha}}.
\end{align*}
We choose $n_2:=\lfloor \log(Nx)/((p+\eta)\log 2)  \rfloor$, which yields us to
$2^{(n_2-1)\alpha}\geq (Nx)^{\alpha/(p+\eta)}/2^{2\alpha}$ and 
$(Nx) 2^{(\alpha-p-\eta)n_2}\leq (Nx)^{\alpha/(p+\eta)}$.
Consequently (recall that $x\in(0,A]$),
\begin{align*}
J^2(N,x)\leq & C (1+\log(Nx)+N)  \exp(-a (Nx)^{\alpha/(p+\eta)})
\leq C(1+N)  \exp(-a (Nx)^{\alpha/(p+\eta)}).
\end{align*}
For any fixed $\e\in(0,\alpha)$, we choose $\eta>0$ small enough so that 
$\alpha/(p+\eta)\geq (\alpha-\e)/p$ and
we conclude that (recall that $Nx\geq 1/A$ because $Nx^2\geq 1$ and $x\leq A$)
\begin{align*}
J^2(N,x)\leq & C(1+N)  \exp(-a (Nx)^{(\alpha-\e)/p}) \leq C  \exp(-a (Nx)^{(\alpha-\e)/p}).
\end{align*}
The last inequality is easily checked, using that  $Nx^2\geq 1$ implies that $N\leq (Nx)^2$.

\vip

{\bf Step 3: bound of $J_n$ under \eqref{asmom}.} Here we show that
for all $\e\in(0,q)$, if $\eta>0$ is small enough, 
$$
\sum_{n\geq 0}J_n(N,x)\leq C N \left(\frac1{Nx}\right)^{(q-\e)/p} \hbox{ if $Nx \geq 1$}.
$$
As usual, this will imply that for all $x>0$, all $N\geq 1$,
$$
\left(\sum_{n\geq 0}J_n(N,x)\right)\land 1
\leq C N \left(\frac1{Nx}\right)^{(q-\e)/p}.
$$

\vip

Exactly as in Step 2, we get from Lemma \ref{bin}-(a)-(b) that
\begin{align*}
J_n(N,x) \leq &
\min \left\{\exp\left(-a N \mu(B_n)z_n \log [2\lor z_n]\right),  N\mu(B_n)\right\}.
\end{align*}
Hence for $n_3$ to be chosen later, since $a N \mu(B_n)z_n= a Nx2^{-(p+\eta)n}$,
\begin{align*}
\sum_{n\geq 0}J_n(N,x)\leq& C\sum_{n=0}^{n_3} \exp(- a Nx2^{-(p+\eta)n}) + C N \sum_{n>n_3} 2^{-qn}\\
\leq & Cn_3 \exp(-a Nx2^{-(p+\eta)n_3}) + C N 2^{-qn_3}.
\end{align*}
We choose $n_3:= \lfloor (q-\e)\log(Nx)/(pq\log 2) \rfloor$, which implies that
$2^{-qn_3} \leq 2^q (Nx)^{-(q-\e)/p}$ and that $2^{-(p+\eta)n_3} \geq (Nx)^{-(q-\e)(p+\eta)/(pq)}$. Hence
\begin{align*}
\sum_{n\geq 0}J_n(N,x)\leq& C \log(Nx) \exp(-a(Nx)^{1-(q-\e)(p+\eta)/(pq)} )
+CN (Nx)^{-(q-\e)/p}.
\end{align*}
If $\eta \in (0,p\e/(q-\e))$, then $1-(q-\e)(p+\eta)/(pq)>0$, and thus
$$\log(Nx) \exp(-a(Nx)^{1-(q-\e)(p+\eta)/(pq)} ) \leq C (Nx)^{-(q-\e)/p}.$$ 
This ends the step.

\vip

{\bf Step 4.} We next assume \eqref{asexfort}
and prove that for all $x\geq A_1:=2^{p}[M_p(\mu)+(2 \log \cE_{\alpha,1}(\mu))^{p/\alpha}]$,
$$
\Pr[Z_N^{p} \geq x ] \leq C \exp(-a N x^{\alpha/p}).
$$ 

A simple computation shows that for any $\nu \in \cP(\rd)$,
$\sum_{n\geq 0} 2^{pn}\nu(B_n) \leq 2^p M_p(\nu)$, whence
$Z_N^{p} \leq 2^p M_p(\mu)+ 2^p N^{-1}\sum_1^N |X_i|^p\leq 
2^p M_p(\mu)+ 2^p [N^{-1}\sum_1^N |X_i|^\alpha]^{p/\alpha}$. Thus
$$
\Pr[Z_N^{p} \geq x ] \leq \Pr\left[N^{-1}\sum_1^N |X_i|^\alpha \geq [x2^{-p} - M_p(\mu)]^{\alpha/p}\right].
$$
Next, we note that for $y\geq 2 \log \cE_{\alpha,1}(\mu)$,
$$
\Pr\left[N^{-1}\sum_1^N |X_i|^\alpha \geq y\right] 
\leq \exp(-N y + N \log\cE_{\alpha,1}(\mu) ) \leq \exp(-N y/2).
$$
The conclusion easily follows, since $x\geq A_1$ implies that 
$y:=[x2^{-p} - M_p(\mu)]^{\alpha/p}\geq 2 \log \cE_{\alpha,1}(\mu)$
and since $y \geq [x2^{-p-1}]^{\alpha/p}-[M_p(\mu)]^{\alpha/p}$.

\vip

{\bf Step 5.} Assume \eqref{asexfaible} and put $\delta:= 2p/\alpha -1$.
Here we show that for all $x>0$, $N\geq 1$, 
$$
\Pr[Z_N^{p} \geq x ] \leq C \exp(-a(Nx)^{\alpha/p})+C \exp(-a Nx^2 (\log (1+N))^{-\delta}).
$$

{\it Step 5.1.} For $R>0$ (large) to be chosen later, we introduce the probability measure $\mu^R$
as the law of $X\indiq_{\{|X|\leq R\}}$. We also denote by $\mu_N^R$ the corresponding empirical
measure (coupled with $\mu_N$ in that the $X_i$'s are used for $\mu_N$ and the 
$X_i\indiq_{\{|X_i|\leq R\}}$'s are chosen for $\mu_N^R$).
We set $Z_N^{p,R}:=\sum_{n\geq 0}2^{pn}|\mu_N^R(B_n)-\mu^R(B_n)|$ and first observe that
$|Z^p_N-Z^{p,R}_N| \leq 2^p N^{-1}\sum_1^N |X_i|^p\indiq_{\{|X_i|>R\}} + 2^p \int_{\{|x|>R\}}|x|^p\mu(dx)$.
On the one hand, $\int_{\{|x|>R\}}|x|^p\mu(dx)\leq \exp(-R^\alpha/2) \int |x|^p e^{|x|^\alpha/2}
\mu(dx) \leq C \exp(-R^\alpha/2)$ by \eqref{asexfaible} (with $\gamma=1$). 
On the other hand, since $\alpha \in (0,p]$,  $\sum_1^N |X_i|^p\indiq_{\{|X_i|>R\}}\leq 
(\sum_1^N |X_i|^\alpha\indiq_{\{|X_i|>R\}})^{p/\alpha}$. Hence if 
$x\geq A \exp(-R^\alpha/2)$, where $A:=2^{p+1}C$,
\begin{align*}
\Pr\left(|Z^p_N-Z_N^{p,R}|\geq x\right)\leq& \Pr\left(N^{-1}\sum_1^N |X_i|^p\indiq_{\{|X_i|>R\}}
\geq x2^{-p-1}  \right)\\
\leq & \Pr\left(\sum_1^N |X_i|^\alpha\indiq_{\{|X_i|>R\}}\geq  (Nx2^{-p-1})^{\alpha/p} \right)\\
\leq & \exp(-  (Nx2^{-p-1})^{\alpha/p}/2)\E\left[\exp\left(|X_1|^\alpha\indiq_{\{|X_1|>R\}}/2\right)
\right]^N.
\end{align*}
Observing that $\E[\exp(|X_1|^\alpha\indiq_{\{|X_1|>R\}}/2)]\leq 1+ 
\E[\exp(|X_1|^\alpha/2)\indiq_{\{|X_1|>R\}}] \leq 1 + C\exp(-R^\alpha/2)$ by \eqref{asexfaible}
and using that $\log(1+u)\leq u$, we deduce that for all $x\geq 2^{p+1} C \exp(-R^\alpha/2)$,
\begin{align*}
\Pr\left(|Z^p_N-Z_N^{p,R}|\geq x\right)
\leq & \exp\left(-  (Nx2^{-p-1})^{\alpha/p}/2+ CN \exp(-R^\alpha/2) \right).
\end{align*}
With the choice
\begin{equation}\label{dfR}
R:= (2\log (1+N))^{1/\alpha},
\end{equation}
we finally find 
\begin{align*}
\Pr\left(|Z^p_N-Z_N^{p,R}|\geq x\right)
\leq & \exp\left(-  (Nx2^{-p-1})^{\alpha/p}/2 + C \right) \leq C \exp\left(-  a(Nx)^{\alpha/p}\right)
\end{align*}
provided $x\geq A \exp(-R^\alpha/2)$, i.e. $(N+1)x \geq A$. As usual, this immediately extends
to any value of $x>0$.

\vip

{\it Step 5.2.} To study $Z_N^{p,R}$, we first observe that since $\mu^R(B_n)=0$ if $2^{n-1}\geq R$,
we have $2^{pn}\mu^R(B_n) \leq  (2R)^{p-\alpha/2}2^{\alpha n/2}\mu^R(B_n)$ for all $n\geq 0$.
Hence $Z_N^{p,R}\leq (2R)^{p-\alpha/2}Z_N^{\alpha/2,R}$. But $\mu^R$ satisfies $\intrd \exp(|x|^\alpha)
\mu^R(dx)<\infty$ uniformly in $R$, so that we may use Steps 1, 2 and 4 (with $p=\alpha/2<\alpha$)
to deduce that for all $x>0$, $\Pr(Z_N^{\alpha/2,R}\geq x) \leq C \exp(-a N x^2)$. Consequently,
$\Pr(Z_N^{p,R}\geq x) \leq C \exp(-a N (x/R^{p-\alpha/2})^2)$. Recalling \eqref{dfR}
and that $\delta:= 2p/\alpha -1$, we see that
that $\Pr(Z_N^{p,R}\geq x) \leq C \exp(-a Nx^2 (\log (1+N))^{-\delta})$. This ends the step.

\vip

{\bf Conclusion.} Recall that $x_0>0$ is fixed.

\vip

First assume \eqref{asexfort}. By Step 4,  
$\Pr[Z_N^{p} \geq x ] \leq C\exp(-aNx^{\alpha/p})$ for all $x\geq A_1$. We deduce from
Steps 1 and 2 that for $x\in (0,A_1)$, $\Pr[Z_N^{p} \geq x ]\leq C\exp(-aNx^2)$.
We easily conclude that for all $x>0$, $\Pr[Z_N^{p} \geq x ]\leq C\exp(-aNx^2)\indiq_{\{x\leq x_0\}}
+  C\exp(-aNx^{\alpha/p})\indiq_{\{x>x_0\}}$ as desired.

\vip

Assume next \eqref{asexfaible}. By Step 5,
$\Pr[Z_N^{p} \geq x ] \leq C \exp(-a Nx^2 (\log (1+N))^{-\delta})+C\exp(-a(Nx)^{\alpha/p})$.
But if $x \geq x_0$, we clearly have $(Nx)^{\alpha/p} \leq C 
Nx^2 (\log (1+N))^{-\delta}$ because $\alpha<p$, 
so that $\Pr[Z_N^{p} \geq x ] \leq C\exp(-a(Nx)^{\alpha/p})$. If now $x \leq x_0$, we use
Steps 1 and 2 to write $\Pr[Z_N^{p} \geq x] \leq C\exp(-aNx^2)+C\exp(-a(Nx)^{(\alpha-\e)/p})$.

\vip

Assume finally \eqref{asmom}. By Steps 1 and 3,
$\Pr[Z_N^{p} \geq x ]\leq C\exp(-aNx^2) + C N (Nx)^{-(q-\e)/q}$
for all $x>0$.
But if $x\geq x_0$, $\exp(-aNx^2)\leq \exp(-aNx)\leq C (Nx)^{-(q-\e)/q}
\leq C N (Nx)^{-(q-\e)/q}$. We conclude that for all $x>0$, 
$\Pr[Z_N^{p} \geq x ]\leq C\exp(-aNx^2)\indiq_{\{x\leq x_0\}} + C N (Nx)^{-(q-\e)/q}$ as desired.
\end{proof}

We can now give the

\begin{preuve} {\it of Theorem \ref{maindev}.}
Using Lemma \ref{fonda}, we write
\begin{align*}
\cT_p(\mu_N,\mu)\leq& \kappa_{p,d} \cD_p(\mu_N,\mu)\\
\leq& \kappa_{p,d} \sum_{n\geq 0} 2^{pn}|\mu_N(B_n)-\mu(B_n)| 
+  \kappa_{p,d}\sum_{n\geq 0} 2^{pn} \mu(B_n)\cD_p(\cR_{B_n}\mu_N,\cR_{B_n}\mu)\\
=:&\kappa_{p,d}( Z_N^p +  V_N^p).
\end{align*}
Hence 
\begin{align*}
\Pr(\cT_p(\mu_N,\mu)\geq x) \leq \Pr(Z_N^p \geq x/(2\kappa_{p,d})) +\Pr(V_N^p \geq x/(2\kappa_{p,d})).
\end{align*}
By Lemma \ref{znpex} (choosing $x_0:=1/(2\kappa_{p,d})$), we easily find 
$\Pr(Z_N^p \geq x/(2\kappa_{p,d})) 
\leq Ce^{-cNx^2}\indiq_{\{x\leq 1\}} +b(N,x) \leq a(N,x)\indiq_{\{x\leq 1\}}+b(N,x)$,
these quantities being defined in the statement of Theorem \ref{maindev}.
We now check that there is $A>0$ such that for all $x>0$,
\begin{align}\label{toprove}
\Pr[V_N^p \geq x/(2\kappa_{p,d}) ] \leq a(N,x)\indiq_{\{x \leq A\}}.
\end{align}
This will end the proof, since one easily checks that 
$a(N,x)\indiq_{\{x \leq A\}}\leq a(N,x)\indiq_{\{x \leq 1\}}+b(N,x)$
(when allowing the values of the constants to change).

\vip

Let us thus check \eqref{toprove}.
For $\eta>0$ to be chosen later, we set $c:=(1-2^{-\eta})/(2\kappa_{p,d})$
and $z_n:= c x 2^{-(p+\eta)n}/\mu(B_n)$. Observing that
$\sum_{n\geq 0} (1-2^{-\eta})2^{-\eta n}=1$), we write
\begin{align*}
\P\left(V_N^p \geq x/(2\kappa_{p,d}) \right)
\leq&\left( \sum_{n\geq 0}
\P\left[ \cD_p(\cR_{B_n}\mu_N,\cR_{B_n}\mu) \geq z_n\right]\right)\land 1
=:  \left(\sum_{n\geq 0} K_n(N,x)\right)\land 1.
\end{align*}
From now on, the value of $c>0$ is not allowed to vary anymore.
We introduce another positive constant $a>0$ whose value may change from line to line.
We only assume \eqref{asmom} (which is implied by \eqref{asexfort} or \eqref{asexfaible}).
We now show that if $\eta>0$ is small enough,
\begin{align}\label{toprove2}
\sum_{n\geq 0} K_n(N,x) \leq C \exp(-a N h(x))\indiq_{\{x \leq A\}} \quad  \hbox{if $Nh(x)\geq 1$,}
\end{align}
where $h(x)=x^2$ if $p>d/2$, $h(x)=(x/\log(2+1/x))^2$ if $p=d/2$ and $h(x)=x^{d/p}$ if
$p<d/2$. 
This will obviously imply as usual that for all $x>0$, 
$$
\left(\sum_{n\geq 0} K_n(N,x)\right)\land 1 \leq C \exp(-a N h(x))\indiq_{\{x\leq A\}}
$$
and thus conclude the proof of \eqref{toprove}. We thus only have to prove \eqref{toprove2}.

\vip

Conditionally on $\mu_N(B_n)$, $\cR_{B_n}\mu_N$ is the empirical measure of $N\mu_N(B_n)$ points
which are $\cR_{B_n}\mu$-distributed. Since $\cR_{B_n}\mu$ is supported in 
$(-1,1]^d$, we may apply Proposition \ref{devcomp} and obtain
\begin{align*}
K_n(N,x)\leq& C\E\left[\indiq_{\{z_n\leq 1\}}
\exp\left(- a N\mu_N(B_n) h(z_n)\right)\right]
\leq C \indiq_{\{z_n\leq 1\}} \exp(-N\mu(B_n)(1- e^{-a h(z_n)}))
\end{align*}
by Lemma \ref{bin}-(c). But the condition $z_n\leq 1$ implies that $h(z_n)$
is bounded (by a constant depending only on $p$ and $d$), whence
$$
K_n(N,x)\leq C\indiq_{\{z_n\leq 1\}} \exp(- a N\mu(B_n)h(z_n)).
$$
By \eqref{asmom}, we have $\mu(B_n)\leq C 2^{-qn}$. Hence 
if $x>A:=C/c$, we have $z_n\geq (c/C)x2^{(q-p-\eta)n} >1$ for all $n\geq 1$ (if 
$\eta\in (0,q-p)$) and thus 
$\sum_{n\geq 0} K_n(N,x)=0$ 
as desired.

\vip

Next, we see that $\theta \mapsto \theta h(x/\theta)$ is decreasing, whence for all $x\leq A$,
$$
K_n(N,x)\leq C \exp(- a N2^{-qn}h(c x 2^{(q-p-\eta)n}/C))\leq  
C \exp(- a N2^{-qn}h(x 2^{(q-p-\eta)n})).
$$
We now treat separately the three cases.

\vip

{\bf Step 1: case $p>d/2$.} Since $h(x)=x^2$, we have, if $\eta \in (0,q/2-p)$,
\begin{align*}
\sum_{n\geq 0} K_n(N,x) 
\leq C \sum_{n\geq 0} \exp(-a N x^2 2^{n(q-2p-2\eta)})\leq C \exp(-a Nx^2)
\end{align*}
if $Nx^2\geq 1$.

{\bf Step 2: case $p=d/2$.} Since $h(x)=(x/\log(2+1/x))^2$, we have, if $\eta \in (0,q/2-p)$,
\begin{align*}
\sum_{n\geq 0} K_n(N,x) \leq& C \sum_{n\geq 0} \exp\left(-a N x^2 2^{(q-2p-2\eta)n}/ 
\log^2[2+1/(x2^{(q-p-\eta)n})]\right)\\
\leq  & C \sum_{n\geq 0} \exp(-a N h(x) 2^{n(q-2p-2\eta)})\\
\leq & C\exp(-a N h(x))
\end{align*}
if $N h(x) \geq 1$. The third inequality only uses that
$\log^2(2+1/(x2^{n(q-p-\eta)})) \leq \log^2(2+1/x)$.

\vip

{\bf Step 3: case $p<d/2$.} Here $h(x)=x^{d/p}$. 
Since $p<d/2$ and $q>2p$, it holds that $q(1-p/d)-p>0$.
We thus may take $\eta \in (0,q(1-p/d)-p)$ (so that $q(d/p-1)-d-d\eta/p>0$)
and we get
\begin{align*}
\sum_{n\geq 0} K_n(N,x) \leq& C \sum_{n\geq 0} \exp(-a N x^{d/p} 2^{n(q(d/p-1)-d-d\eta/p)})
\leq  C\exp(-a N x^{d/p})
\end{align*}
if  $N x^{d/p} \geq 1$.
\end{preuve}

\section{The dependent case}

We finally study a few classes of dependent sequences of random variables. We only
give some moment estimates. Concentration inequalities might be obtained, but this should be
much more complicated.

\subsection{$\rho$-mixing stationary sequences}

A stationary sequence of random variables $(X_n)_{n\ge 1}$ with common law $\mu$ is said to be 
$\rho$-mixing, for some $\rho:\nn\to\rr^+$ with $\rho_n\to 0$, if for all $f,g \in L^2(\mu)$ 
and all $i,j\geq 1$
$$\Cov(f(X_i),g(X_j))\le \rho_{|i-j|}\sqrt{\Var(f(X_i))\Var(g(X_j))}.$$
We refer for example to Rio \cite{rio}, Doukhan \cite{douk} or Bradley \cite{bradbook}. 

\begin{thm}\label{mix}
Consider a stationary sequence of random variables $(X_n)_{n\ge 1}$ with common law $\mu$
and set $\mu_N:=N^{-1}\sum_1^N \delta_{X_i}$.
Assume that this sequence is $\rho$-mixing, for some $\rho:\nn\to\rr^+$ satisfying
$\sum_{n\geq 0} \rho_n<\infty$. Let $p>0$ and assume that $\mu \in M_q(\rd)$ for
some $p>q$. 
There exists a constant $C$ depending only on $p,d,q, M_q(\mu),\rho$ such that, for all $N\geq 1$,
$$
\E\left(\cT_p(\mu_N,\mu)\right) \leq
C\left\{\begin{array}{ll}
N^{-1/2} +N^{-(q-p)/q}& \!\!\!\hbox{if $p>d/2$ and $q\ne 2p$},  \\[+3pt]
N^{-1/2} \log(1+N)+N^{-(q-p)/q} &\!\!\! \hbox{if $p=d/2$ and $q\ne 2p$}, \\[+3pt]
N^{-p/d}+N^{-(q-p)/q} &\!\!\!\hbox{if $p\in (0,d/2)$ and $q\ne d/(d-p)$}.
\end{array}\right.
$$
\end{thm}

This is very satisfying: we get the same estimate as in the independent case. 
The case $\sum_{n\geq 0} \rho_n=\infty$ can also be treated (but then the upper bounds will be 
less good and depend on the rate of decrease of $\rho$). 
Actually, the $\rho$-mixing condition is slightly too strong (we only need 
the covariance inequality when $f=g$ is an indicator function), but it is best adapted 
notion of mixing we found in the litterature.

\begin{proof}
We first check that for any Borel subset $A \subset \rd$,
$$
\E[|\mu_N(A)-\mu(A)|] \leq \min\{2\mu(A),C\mu(A)N^{-1/2}\}.
$$
But this is immediate: $\E[\mu_N(A)]=\mu(A)$ (whence $\E[|\mu_N(A)-\mu(A)|]\leq 2\mu(A)$) and
\begin{align*}
\Var \mu_N(A)=& \frac{1}{N^2} \sum_{i,j\leq N} \Cov(\indiq_A(X_i),\indiq_A(X_i))\\
\leq& \frac{1}{N^2} \sum_{i,j\leq N} \rho_{|i-j|}\Var(\indiq_A(X_1))\\
\leq& \frac{\mu(A)(1-\mu(A))}{N^2}\sum_{i,j\leq N} \rho_{|i-j|}.
\end{align*}
This is smaller than $C \mu(A)/N$ 
as desired, since $\sum_{i,j\leq N} \rho_{|i-j|}\leq N \sum_{k\geq 0} \rho_k= C N$.
Once this is done, it suffices to copy (without any change) the proof of Theorem \ref{mainmom}.
\end{proof}

\subsection{Markov chains}

Here we consider a $\rd$-valued Markov chain $(X_n)_{n \geq 1}$ with transition kernel $P$
and initial distribution $\nu \in \cP(\rd)$ and we set $\mu_N:=N^{-1}\sum_{1}^N \delta_{X_n}$.
We assume that it admits a unique invariant
probability measure $\pi$ and the following $L^2$-decay property 
(usually related to a Poincar\'e inequality)
\begin{align}\label{decay}
\forall \;n\geq 1,\; \forall \;f\in L^2(\pi), 
\quad \|P^nf-\pi(f)\|_{L^2(\pi)}\le C\rho_n \|f-\pi(f)\|_{L^2(\pi)}
\end{align}
for some sequence $(\rho_n)_{n\geq 1}$ decreasing to $0$.

\begin{thm}\label{markov} Let $p\geq 1$, $d\geq 1$ and $r> 2$ be fixed.
Assume that our Markov chain $(X_n)_{n \geq 0}$ satisfies 
\eqref{decay} with a sequence $(\rho_n)_{n\geq 1}$ satisfying $\sum_{n\geq 1} \rho_n<\infty$.
Assume also that the initial distribution $\nu$ is absolutely continuous with respect to $\pi$
and satisfies $\|d\nu/d\pi\|_{L^r(\pi)}<\infty$. Assume finally that $M_q(\pi)<\infty$ for some
$q>p r/(r-1)$. Setting $q_r:=q(r-1)/r$ and $d_r=d(r+1)/r$, there is a constant $C$ such that for
all $N\geq 1$,
$$
\E_\nu\left(\cT_p(\mu_N,\pi)\right) \leq
C\left\{\begin{array}{ll}
N^{-1/2} +N^{-(q_r-p)/q_r}& \hbox{if $p>d_r/2r$ and $q_r\ne 2p$},  \\[+3pt]
N^{-1/2} \log(1+N)+N^{-(q_r-p)/q_r} & \hbox{if $p=d_r/2r$ and $q_r\ne 2p$}, \\[+3pt]
N^{-p/d}+N^{-(q_r-p)/q_r} &\hbox{if $p\in (0,d_r/2)$ and $q_r\ne d_r/(d_r-p)$}.
\end{array}\right.
$$
\end{thm}

Once again, we might adapt the proof to get a complete picture 
corresponding to other decay than $L^2$-$L^2$ and to slower mixing rates $(\rho_n)_{n\geq 1}$.

\begin{proof}
We only have to show that for any $\ell \geq 0$, any $n\geq 0$, 
\begin{align*}
\Delta_{n,\ell}^N:=&\sum_{F\in\cP_\ell}\E_\nu\left( |\mu_N(2^n F \cap B_n)-\pi(2^n F \cap B_n)| \right)\\
\leq& C \min \left\{(\pi(B_n))^{(r-1)/r}, [2^{d_r \ell}(\pi(B_n))^{(r-1)/r}/N]^{1/2} \right\}.
\end{align*}
Since $M_q(\pi)<\infty$ (whence $\pi(B_n)\leq C 2^{-qn}$), we will deduce that
$$
\Delta_{n,\ell}^N
\leq C \min \left\{2^{-q_r n}, 2^{d_r \ell/2}(2^{-q_rn}/N)^{1/2} \right\}.
$$
Then the rest of the proof is exactly the same as that of Theorem \ref{mainmom}, 
replacing everywhere $q$ and $d$ by $q_r$ and $d_r$.

\vip

We first check that $\Delta_{n,\ell}^N \leq  C (\pi(B_n))^{(r-1)/r}$. Using that
$\|d\nu/d\pi\|_{L^r(\pi)}<\infty$, we write 
$$
\E_\nu(\mu_N(B_n))=\frac1N\sum_{i=1}^N\E_\pi\left[\frac{d\nu}{d\pi}(X_0)1_{\{X_i\in B_n\}}\right]
\leq  \|d\nu/d\pi\|_{L^r(\pi)}  \pi(B_n)^{(r-1)/r}.
$$

We next consider a Borel subset $A$ of $\rd$ and check that
$$
\E_\nu(|\mu_N(A)-\pi(A)|)\le C  (\pi(A))^{(r-1)/(2r)} N^{-1/2}.
$$
To do so, as is usual when working with Markov chains or covariance properties (see \cite{blg}), 
we introduce $f=1_A-\pi(A)$ and write
\begin{align*}
\E_\nu(|\mu_N(A)-\pi(A)|)=\frac1N\E_\nu\left(\left|\sum_{i=1}^N f(X_i)\right|\right)
\le \frac1N \left( \sum_{i,j=1}^N\E_\nu(f(X_i)f(X_j))\right)^{1/2}.
\end{align*} 
For $j\geq i$, it holds that
\begin{align*}
\E_\nu(f(X_i)f(X_j))=&\E_\nu[f(X_i)P^{j-i}f(X_i)]
=\E_\pi\left[\frac{d\nu}{d\pi}(X_0)f(X_i).P^{j-i}f(X_i)\right].
\end{align*}
Using the H\"older inequality (recall that $\|d\nu/d\pi\|_{L^r(\pi)}<\infty$ with $r>2$)
and \eqref{decay}, we get
\begin{align*}
\E_\nu(f(X_i)f(X_j))\leq& \|d\nu/d\pi\|_{L^r(\pi)} \|f\|_{L^{2r/(r-2)}(\pi)}\|P^{j-i}f\|_{L^2(\pi)}
\leq C \rho_{j-i} \|f\|_{L^{2r/(r-2)}(\pi)}\|f\|_{L^2(\pi)} .
\end{align*}
But for $s>1$, $\|f\|_{L^s(\pi)}\leq C_s(\pi(A)+(\pi(A))^s)^{1/s}\leq C_s (\pi(A))^{1/s}$,
we find $\E_\nu(f(X_i)f(X_j))\leq C \rho_{j-i}(\pi(A))^{(r-1)/r}$ and thus
\begin{align*}
\E_\nu(|\mu_N(F)-\pi(F)|)\le \frac CN \left(\sum_{i,j=1}^N \rho_{|i-j|} (\pi(F))^{(r-1)/2r} \right)^{1/2}
\leq C  (\pi(F))^{(r-1)/(2r)} N^{-1/2}
\end{align*} 
as desired. We used that $\sum_{i,j=1}^N \rho_{|i-j|}\leq C N$.

\vip

We can finally conclude that
\begin{align*}
\Delta_{n,\ell}^N \leq & CN^{-1/2} \sum_{F \in \cP_\ell} (\pi(2^n F \cap B_n))^{(r-1)/(2r)}
\leq  CN^{-1/2} 2^{d_r \ell / 2} (\pi(B_n))^{(r-1)/(2r)}
\end{align*}
by the H\"older inequality (and because $\# \cP_\ell=2^{d\ell}$), where $d_r=d (r+1)/r$ 
as in the statement.
\end{proof}

\subsection{Mc Kean-Vlasov particles systems}
Particle approximation of nonlinear equations has attracted a lot of attention in the past
thirty years. We will focus here on the following $\rd$-valued nonlinear S.D.E.
$$
dX_t=\sqrt{2}dB_t-\nabla V(X_t)dt-\nabla W*u_t(X_t)dt,\qquad X_0=x
$$
where $u_t= Law(X_t)$ and $(B_t)$ is and $\rd$-Brownian motion. This is a probabilistic 
representation of the so-called Mc Kean-Vlasov equation, which has been studied  
in particular by Carillo-Mac Cann-Villani \cite{cmcv}, Malrieu \cite{m} and Cattiaux-Guillin-Malrieu
\cite{cgm} to which we refer for further motivations and existence and uniqueness of
solutions. We will mainly consider here the case where $V$ and $W$ are convex (and if $V=0$ the 
center of mass is fixed) and $W$ is even.
To fix the ideas, let us consider only two cases:

(a) $Hess\, V\ge\beta Id>0$, $Hess\, W\ge 0$.

(b) $V(x)=|x|^\alpha$ for $\alpha>2$, $Hess\, W\ge 0$.

\vip

The particle system introduced to approximate the nonlinear equation is the following. 
Let $(B^i_t)_{t\geq 0}$ be $N$ independent Brownian motion. For $i=1,\dots,N$, set $X^{i,N}_0=x$ and
$$dX^{i,N}_t=\sqrt{2}dB^i_t-\nabla V(X^{i,N}_t)dt-\frac1N\sum_j\nabla W(X^{i,N}_t-X^{j,N}_t)dt.$$
Usual propagation of chaos property is usually concerned with control of
$$\cW_2(Law(X^{1,N}_t),u_t)$$
uniformly (or not) in time. It is however very natural to consider rather a control of
$$\cW_2(\hat u^N_t,u_t)$$
where $\hat u^N_t=\frac 1N\sum_{i=1}^N\delta_{X^{i,N}_t}$, as in Bolley-Guillin-Villani \cite{bgv}.

\vip

To do so, and inspired by the usual proof of propagation of chaos, let us consider nonlinear 
independent particles
$$
dX^i_t=\sqrt{2}dB^i_t-\nabla V(X_t^i)dt-\nabla W*u_t(X_t^i)dt,\qquad X^i_0=x
$$
(driven by the same Brownian motions as the particle system)
and the corresponding empirical measure $u^N_t=\frac 1N\sum_{i=1}^N\delta_{X^{i}_t}$. We then have
$$
\cW_2(\hat u^N_t,u_t)\le \cW_2(\hat u^N_t,u^N_t)+\cW_2(u^N_t,u_t).
$$
Then following \cite{m} in case (a) and \cite{cgm} in case (b), one easily gets (for
some time-independent constant $C$)
$$
\E(\cW_2^2(\hat u^N_t,u^N_t))\le \frac1N\E\left(\sum_{i=1}^N|X^{i,N}_t-X^i_t|^2\right)\le C\alpha(N)
$$
where $\alpha(n)=N^{-1}$ in case (a) and $\alpha(N)=N^{-1/(\alpha-1)}$ in case (b). It is not hard to 
prove here that the nonlinear particles have infinitely many moments (uniformly in time)
so that combining Theorem \ref{mainmom} with the previous estimates gives
$$
\sup_{t\geq 0}\E(\cW_2(\hat u^N_t,u_t))\le C(\alpha(N)+\beta(N))
$$
where $\beta(N)=N^{-1/2}$ if $d=1$, $\beta(N)=N^{-1/2}\log(1+N)$ if $d=2$ and
$\beta(N)=N^{-1/d}$ if $d\geq 3$.


\end{document}